\documentclass[11pt, leqno]{amsart}

\usepackage{verbatim}
\usepackage{etex}
\usepackage{pictexwd}
\usepackage{lscape}
\usepackage{graphicx}
\usepackage{amssymb,amsmath,amsthm,amscd, amsfonts}
\usepackage{mathrsfs}
\usepackage{enumerate}
\usepackage[usenames,dvipsnames]{color}
\usepackage[vcentermath,enableskew,noautoscale]{youngtab}
\usepackage[all]{xy}
\usepackage{ytableau}

\allowdisplaybreaks[4]

\usepackage[normalem]{ulem}  
\usepackage{tikz}
\usetikzlibrary{shapes,decorations,arrows,calc,arrows.meta,fit,positioning}

\tikzset{
    -Latex,auto,node distance =1 cm and 1 cm,semithick ,
    state/.style ={ellipse, draw, minimum width = 0.7 cm},
    point/.style = {circle, draw, inner sep=0.05cm,fill,node contents={}},
    bidirected/.style={Latex-Latex,dashed},
    el/.style = {inner sep=2pt, align=left, sloped}
}

\makeatletter
\newcommand{\xleftrightarrow}[2][]{\ext@arrow 3359\leftrightarrowfill@{#1}{#2}}
\newcommand{\xdashrightarrow}[2][]{\ext@arrow 0359\rightarrowfill@@{#1}{#2}}
\newcommand{\xdashleftarrow}[2][]{\ext@arrow 3095\leftarrowfill@@{#1}{#2}}
\newcommand{\xdashleftrightarrow}[2][]{\ext@arrow 3359\leftrightarrowfill@@{#1}{#2}}
\def\rightarrowfill@@{\arrowfill@@\relax\relbar\rightarrow}
\def\leftarrowfill@@{\arrowfill@@\leftarrow\relbar\relax}
\def\leftrightarrowfill@@{\arrowfill@@\leftarrow\relbar\rightarrow}
\def\arrowfill@@#1#2#3#4{%
  $\m@th\thickmuskip0mu\medmuskip\thickmuskip\thinmuskip\thickmuskip
   \relax#4#1
   \xleaders\hbox{$#4#2$}\hfill
   #3$%
}
\makeatother

\newcommand{\nc}{\newcommand}
\numberwithin{equation}{section}

\newenvironment{red}{\relax\color{red}}{\relax}
\newenvironment{blue}{\relax\color{blue}}{\hspace*{.5ex}\relax}
\newenvironment{jaune}{\relax\color{Orchid}}{\hspace*{.5ex}\relax}

\newcommand{\bj}{\begin{jaune}}
\newcommand{\ej}{\end{jaune}}

\newcommand{\ber}{\begin{red}}
\newcommand{\er}{\end{red}}
\newcommand{\bebl}{\begin{blue}}
\newcommand{\ebl}{\end{blue}}

\theoremstyle{plain}
\newtheorem{lemma}{Lemma}[section]
\newtheorem{prop}[lemma]{Proposition}
\newtheorem{theorem}[lemma]{Theorem}

\newcommand{\Prop}{\begin{prop}}
\newcommand{\enprop}{\end{prop}}
\newcommand{\Lemma}{\begin{lemma}}
\newcommand{\enlemma}{\end{lemma}}
\newcommand{\Th}{\begin{theorem}}
\newcommand{\enth}{\end{theorem}}
\newtheorem{corollary}[lemma]{Corollary}
\newcommand{\Cor}{\begin{corollary}}
\newcommand{\encor}{\end{corollary}}
\newtheorem{definition}[lemma]{Definition}
\newtheorem{conjecture}[lemma]{Conjecture}
\newcommand{\Def}{\begin{definition}}
\newcommand{\edf}{\end{definition}}
\newtheorem{sublemma}[lemma]{Sublemma}
\newcommand{\Sublemma}{\begin{sublemma}}
\newcommand{\ensub}{\end{sublemma}}

\theoremstyle{definition}
\newtheorem{remark}[lemma]{Remark}
\newtheorem{example}[lemma]{Example}
\newtheorem{Convention}[lemma]{Convention}
\newcommand{\Conv}{\begin{Convention}}
\newcommand{\enconv}{\end{Convention}}
\nc{\Con}{\begin{conjecture}}
\nc{\encon}{\end{conjecture}}
\nc{\Rem}{\begin{remark}}
\nc{\enrem}{\end{remark}}
\newcommand{\C}{{\mathbb C}}

\newcommand{\Z}{{\mathbb Z}}
\newcommand{\B}{{\mathbf{B}}}

\newcommand{\seteq}{\mathbin{:=}}

\newcommand{\g}{{\mathfrak{g}}}

\newcommand{\Hom}{\operatorname{Hom}}
\newcommand{\End}{\operatorname{End}}

\newcommand{\isoto}[1][]{\mathop{\xrightarrow%
[{\raisebox{.3ex}[0ex][.3ex]{$\scriptstyle{#1}$}}]%
{{\raisebox{-.6ex}[0ex][-.6ex]{$\mspace{2mu}\sim\mspace{2mu}$}}}}}
\newcommand{\tensor}{\otimes}

\newcommand{\eq}{\begin{eqnarray}}
\newcommand{\eneq}{\end{eqnarray}}

\newcommand{\eqn}{\begin{eqnarray*}}
\newcommand{\eneqn}{\end{eqnarray*}}
\newcommand{\on}{\operatorname}
\newcommand{\Ker}{\on{Ker}}

\newcommand{\QED}{\end{proof}}
\newcommand{\Proof}{\begin{proof}}

\newcommand{\soplus}{\mathop{\mbox{\normalsize$\bigoplus$}}\limits}

\newcommand{\ba}{\begin{array}}
\newcommand{\ea}{\end{array}}

\newcommand{\set}[2]{\left\{#1 \mid #2 \right\}}

\newcommand{\eqsub}{\begin{subequations}\begin{eqnarray}}
\newcommand{\eneqsub}{\end{eqnarray}\end{subequations}}

\newcommand{\ol}{\overline}

\nc{\la}{\lambda}
\nc{\lam}{\lambda}
\nc{\U}[1][\g]{U_q(#1)}
\nc{\te}{\tilde{e}}
\nc{\tei}{\tilde{e}_i}
\nc{\tf}{\tilde{f}}
\nc{\tfi}{\tilde{f}_i}
\nc{\tU}{\widetilde U_q(\g)}
\nc{\tE}{\tilde{E}}
\nc{\tF}{\widetilde{\F}}
\nc{\tK}{\widetilde{K}}

\nc{\tk}{\tilde{k}}
\nc{\tkone}{\tk_{\ol{1}}}
\nc{\teone}{\tilde{e}_{\ol{1}}}
\nc{\tfone}{\tilde{f}_{\ol{1}}}

\nc{\teibar}{\tilde{e}_{\ol{i}}} \nc{\tfibar}{\tilde{f}_{\ol{i}}}
\nc{\tki}{{\tk}_{\ol {i}}}

\nc{\BZ}{{\mathbb{Z}}}
\nc{\al}{\alpha}
\nc{\qs}{{q}}
\nc{\lan}{\langle}
\nc{\ran}{\rangle}
\nc{\re}{{\mathrm{re}}}
\nc{\wt}{\operatorname{wt}}
\nc{\ch}{\operatorname{char}}
\nc{\Um}[1][\g]{U^-_q(#1)}
\nc{\Ue}{U^+_q(\g)}
\nc{\eps}{\varepsilon}
\nc{\vphi}{\varphi}
\nc{\sphi}{\varphi^*}
\nc{\seps}{\varepsilon^*}

\nc{\nn}{\nonumber}

\nc{\vp}{\varpi}
\nc{\cls}{{\operatorname{cl}}}
\nc{\Wt}{{\operatorname{Wt}}}
\nc{\Us}{U'_q(\g)}
\nc{\La}{\Lambda}
\nc{\tLa}{\widetilde\Lambda}
\nc{\ro}{{\rm(}}
\nc{\rf}{{\rm)}}
\nc{\norm}{{\mathrm{norm}}}
\nc{\qbox}{\quad\mbox}
\nc{\braid}{{\mathfrak{B}}}
\nc{\Ad}{\operatorname{Ad}}
\nc{\Aut}{\operatorname{Aut}}
\nc{\dt}[1]{\tilde{\tilde #1}}
\nc{\Sn}{S^{{\mathrm{norm}}}}
\nc{\aff}{{\rm{aff}}}
\nc{\rk}{{\mathrm{rk}}}
\nc{\tP}{\widetilde{P}}
\nc{\tW}{\widetilde{W}}
\nc{\Dyn}{\mathrm{Dyn}}
\nc{\tD}{\widetilde{\Delta}}
\nc{\height}[1]{{\operatorname{ht}}(#1)}
\nc{\bl}{\bigl(}
\nc{\br}{\bigr)}
\nc{\Hecke}{\mathrm{H}}
\nc{\HA}{\Hecke^{\mathrm{A}}}
\nc{\HB}{\Hecke^{\mathrm{B}}}
\newcommand{\scbul}{{\,\raise1pt\hbox{$\scriptscriptstyle\bullet$}\,}}
\nc{\vac}{{\phi}}
\nc{\Bt}{\B_\theta(\g)}
\nc{\be}{\begin{enumerate}}
\nc{\ee}{\end{enumerate}}
\nc{\low}{{\mathrm{low}}}
\nc{\upper}{{\mathrm{up}}}
\nc{\Zodd}{\Z_{\mathrm{odd}}}
\nc{\Ft}[1][n]{\mathbb{P}\mathrm{ol}_{#1}}
\nc{\Ftf}[1][n]{\widetilde{\mathbb{P}\mathrm{ol}}_{#1}}
\nc{\KA}{\on{K}^{\mathrm{A}}}
\nc{\KB}{\on{K}^{\mathrm{B}}}
\nc{\Res}{\on{Res}}
\nc{\Fc}[1][{n,m}]{\mathbf{F}_{#1}}
\nc{\tphi}{\tilde{\varphi}}
\nc{\CO}{\mathscr{O}}
\nc{\inte}{\mathrm{int}}
\nc{\Oint}{\mathcal{O}^{\ge0}_{\inte}}
\nc{\vs}{\vspace*}
\nc{\tLt}{\widetilde{L}}
\nc{\tL}{\widetilde{\Lambda}}
\nc{\tu}{\tilde{u}}
\nc{\noi}{\noindent}
\nc{\heigh}{\mathfrak{t}}
\nc{\lowest}{\mathfrak{l}}
\nc{\rootl}{\mathsf{Q}}
\nc{\cl}{\colon}
\nc{\uqpg}{U'_q(\mathfrak g)}
\nc{\uq}{\uqpg}
\nc{\Oh}{\widehat{\mathcal{O}}}
\nc{\pn}{p_{\mathfrak{n}}}

\nc{\KLR}{KLR algebra}
\nc{\KLRs}{KLR algebras}
\nc{\cor}{\mathbf{k}}
\nc{\cora}{{\cor(A)}}
\nc{\haut}{\mathrm{ht}}
\nc{\tens}{\mathop\otimes}
\nc{\gmod}{\mbox{-$\mathrm{gmod}$}}
\nc{\gMod}{\mbox{-$\mathrm{gMod}$}}
\nc{\proj}{\mbox{-$\mathrm{proj}$}}
\nc{\gproj}{\mbox{-$\mathrm{gproj}$}}
\nc{\smod}{\mbox{-$\mathrm{mod}$}}
\nc{\Mod}{\mbox{-$\mathrm{Mod}$}}
\nc{\h}{\mathfrak h}
\nc{\Rnorm}{R^{\rm{norm}}}

\nc{\Vhat}{\widehat{V}}
\nc{\F}{\mathcal{F}}

\def\T{{\mathcal T}}

\nc{\fd}[1][A]{\on{\mathrm{flat.dim}_{#1}}}
\nc{\bP}{{\mathbb{P}}}
\nc{\bPh}{\widehat{\mathbb{P}}}
\nc{\bK}[1][{n}]{\widehat{\mathbb{K}}_{#1}}
\nc{\bV}[1][{n}]{\widehat{V}^{\otimes{#1}}}
\nc{\bVK}[1][{n}]{\widehat{V}^{\otimes{#1}}_{\widehat{\mathbb{K}}}}
\nc{\hV}{\widehat{V}}
\nc{\opp}{\mathrm{op}}
\nc{\col}{\colon}
\nc{\bnum}{\be[{\rm(i)}]}
\nc{\oep}{\epsilon}
\nc{\qtext}{\quad\text}
\nc{\qtextq}[1]{\quad\text{#1}\quad}
\nc{\longtwoheadrightarrow}[1][]{\xymatrix{\ar@{->>}[r]^-{{#1}}&}}
\nc{\epiTo}[1][]{\longtwoheadrightarrow[{#1}]}
\nc{\epito}{\twoheadrightarrow}
\nc{\monoTo}[1][]{\xymatrix{\ar@{>->}[r]^-{{#1}}&}}
\nc{\sym}{\mathfrak{S}}
\nc{\inp}[1]{{({#1})_{\mathrm{n}}}}
\nc{\rtl}{\rootl}
\nc{\wtd}{\widetilde}
\nc{\etens}{\boxtimes}
\nc{\ds}[1]{\mathrm{d}(#1)}
\nc{\rmat}[1]{{\mathbf{r}}_%
{\mspace{-2mu}\raisebox{-.6ex}{${\scriptstyle{#1}}$}}}
\nc{\rmats}[1]{{\mathbf{r}}_%
{\mspace{-2mu}\raisebox{-.6ex}{${\scriptscriptstyle{#1}}$}}}
\nc{\shc}{\mathcal{C}}
\nc{\shs}{\mathcal{S}}
\nc{\Fct}{{\on{Fct}}}
\nc{\tC}{\widetilde{\shc}}
\nc{\Zp}{\Z_{\ge0}}
\nc{\tPhi}{\widetilde{\Phi}}
\nc{\tT}{{\widetilde{\T}}}
\nc{\Ob}{\on{Ob}}
\nc{\bwr}{\mbox{\large$\wr$}}
\nc{\Img}{\on{Im}}
\nc{\Ab}{\mathcal{A}^{\mathrm{big}}}
\nc{\Sb}{\mathcal{S}^{\mathrm{big}}}
\nc{\As}{\mathcal{A}}
\nc{\Ss}{\mathcal{S}}
\nc{\ntens}{\widetilde{\otimes}}
\nc{\hR}{\widehat{R}}
\nc{\nconv}{\mathop{\mbox{\large $\odot$}}}
\nc{\snconv}{\mbox{\scriptsize$\odot$}}
\nc{\ts}{\tilde{s}}
\nc{\sho}{\mathcal{O}}
\nc{\bc}{\begin{cases}}
\nc{\ec}{\end{cases}}
\nc{\slnh}{{\widehat{\mathfrak{sl}}_N}}
\nc{\UA}{U_q'(\slnh)}
\nc{\KR}{R_K}
\nc{\cQ}{\mathcal{Q}}
\nc{\Irr}{\mathcal{I}rr}
\nc{\tQ}{\widetilde{\cQ}}
\nc{\bs}{\mathbf{s}}
\nc{\bL}{\mathbb{L}}
\nc{\tg}{\tilde{g}}

\nc{\conv}{\mathbin{\mbox{\large $\circ$}}}
\nc{\shconv}{\mathbin{\large\diamond}}
\nc{\hconv}{\mathbin{\mbox{\Large $\shconv$}}}

\nc{\Rm}{R^{\mathrm{ren}}}

\nc{\bQ}{\ol{Q}}
\renewcommand{\Im}{\on{Im}}

\nc{\de}{\on{\textfrak{d}}}

\nc{\xmono}{\ar@{>->}}
\nc{\xepi}{\ar@{->>}}
\nc{\db}[1]{\raisebox{-.5ex}[2ex][1.8ex]{$#1$}}
\nc{\wb}[1]{\mbox{$\rule[-1.1ex]{0ex}{2ex}#1$}}
\nc{\univ}{\mathrm{univ}}
\nc{\rM}{{}^*\mspace{-2mu}M}
\nc{\lM}{M^*}
\nc{\uqm}{\uq\smod}
\nc{\tR}{\widetilde{R}_{\gamma,\beta}}
\nc{\tx}{\tilde{x}}
\nc{\bi}{\mathbf{i}}
\nc{\ttau}{\widetilde{\tau}}

\newcommand{\Par}{\mathcal{A}}
\newcommand{\Br}{\mathcal{B}}
\newcommand{\LG}{\mathcal{L}}

\nc{\tEnd}{\on{\widetilde{E}nd}}
\nc{\tHom}{\on{\widetilde{H}om}}

\nc{\K}{{J}}
\nc{\Kex}{{\K}_{\mathrm{ex}}}
\nc{\Kfr}{{\K}_{\mathrm{f\mspace{.01mu}r}}}
\nc{\coro}{\cor}
\nc{\tB}{\widetilde{B}}
\nc{\seed}{\mathscr{S}}

\nc{\up}{\mathrm{up}}
\nc{\bfa}{\mathbf{a}}

\newcommand{\ULG}{\mathcal{U}} 

\newlength{\mylength}
\nc{\GCT}{\mathcal E}

\title
{Polynomial invariants on matrices and partition,  Brauer algebras}

\author[Myungho Kim, Doyun Koo]{Myungho Kim, Doyun Koo}

\address{Myungho Kim : Department of Mathematics, Kyung Hee University, Seoul 02447, Korea}
\email{mkim@khu.ac.kr}

\address{Doyun Koo : Department of Mathematics, Kyung Hee University, Seoul 02447, Korea}
\email{doyun9@khu.ac.kr}

\thanks{
This research was  supported by the National Research Foundation of
Korea(NF) grant funded by the Korea government(MSIP) (No. NRF-2017R1C1B2007824).}


\date{January 18, 2021} 

\begin{document}

\begin{abstract}
We identify the dimension of the centralizer of the symmetric group $\sym_d$ in the partition algebra $\Par_d(\delta)$ and in the Brauer algebra $\Br_d(\delta)$ with the number of multidigraphs with $d$ arrows and the number of disjoint union of directed cycles with $d$ arrows, respectively. 
Using Schur-Weyl duality as a fundamental theory, we conclude that each centralizer is related to the $G$-invariant space $P^d(M_n(\cor))^G$ of degree $d$ homogeneous polynomials on $n\times n$ matrices,  where $G$ is the orthogonal group and the group of permutation matrices, respectively. 
Our approach gives a uniform way to show that the dimensions of $P^d(M_n(\cor))^G$ are stable for sufficiently large $n$.
\end{abstract}

\maketitle


\nc{\Cent}{\mathscr C}
\section{Introduction}

\vskip 1em
Throughout the paper, $\cor$ is an  infinite field.
 The characteristic of  $\cor$ is denoted by $\ch (\cor)$. 
 For a group $G$ and a $G$-set $X$, $X^G$ denotes the set of fixed points in $X$ under the action of $G$. 
Let $P(M_n(\cor))$ be the space of polynomial functions  on the set $M_n(\cor)$ of $n \times n$ matrices over a field $\cor$ and let $P^d(M_n(\cor))$ be the space of homogeneous polynomial functions of degree $d$.   
Our primary interest in this paper is the dimension of the space  $P^d(M_n(\cor))^G$ of polynomial invariants under the conjugation action of a subgroup $G$ of the general linear group $GL_n(\cor)$. 
An interesting result of Willenbring (\cite{WillenbringO}) along this line is that  the dimensions of the space $P^d(M_n(\C))^{O_n(\C)}$, where $O_n(\C)$ is the orthogonal group over $\C$, stabilize for  $n$ which are greater than or equal to  $d$.
In particular, there exists a \emph{stable limit} of  the Hilbert series of the rings $\{P(M_n(\C))^{O_n(\C)}\}_{n\ge 1}$:
$$\lim HS (t) :=\sum_{d=0}^{\infty}  h_d t^d
=\sum_{d=0}^{\infty}  \Big[ \lim_{n \rightarrow \infty} \dim_\C P^d(M_n(\C))^{O_n(\C)} \Big] \, t^d.$$
It is also shown in \cite{WillenbringO} that
the coefficient $h_d$ in $\lim HS (t)$  is equal to the size of the set $\ULG_d^O$ of directed graphs  with $d$ arrows,  whose connected components are directed cycles. 

In \cite{Willenbring}, this stability phenomenon is extended to a large family of invariant rings arising from \emph{classical symmetric pairs} (for the precise definition, see \cite[Section 1.1]{Willenbring}). 
Among those, the invariant ring $P(M_{2m}(\cor))^{Sp_{2m}(\C)}$ appears, where $Sp_{2m}(\C)$ denotes the symplectic group.  
Moreover,  the stable Hilbert series for the family $\{P(M_{2m}(\C))^{Sp_{2m}(\C)}\}_{m\ge 1}$ is equal to $\lim HS(t)$ above.
This stability and the equality on the stable Hilbert series are consequences of a description of  the coefficients in the Hilbert series in terms of Littlewood-Richardson coefficients.  Note that these descriptions are  obtained by a case-by-case investigation for each group (for detail, see \cite[Section 1.7]{Willenbring}).

\medskip
In this paper, we provide a new approach to explain the stability of the Hilbert series for the rings $P(M_n(\cor))^{G_n}$ of some family of subgroups  $G_n$ of $GL_n(\cor)$. Our approach includes the cases that $\cor$ is $\C$ and $G_n$ is either $O_n(\C)$ or $Sp_{2m}(\C)$.
One of the starting points is the following well-known fact:  the space $P^d(M_n(\cor))$ is dual to the \emph{Schur algebra} $S_{n,d}(\cor)$, the image of the $GL_n(\cor)$ action on the tensor power $V^{\tens d}$, where $V=\cor^n$. 
And the Schur algebra $S_{n,d}(\cor)$ can be identified with the centralizer of the symmetric group action on $V^{\tens d}$ by Schur-Weyl duality for $GL_n(\cor)$.
 Based on these facts,   one can identify the dimension of   the space $P^d(M_n(\cor))^{G_n}$ with the dimension of the
  centralizer
  of the  symmetric group $\sym_d$ action inside the algebra of $G_n$-module endomorphisms on   $V^{\tens d}$, provided either $\ch(\cor) =0 $ or $\ch(\cor) > d $ (Theorem \ref{thm:dim}). That is,
 we have the following equality.  
  \eq \label{eq:dimGn} 
  \dim_\cor P^d(M_n(\cor))^{G_n}  = \dim_\cor \Cent_{\End_{G_n}(V^{\tens d})}\left( \Psi_{n,d}(\cor \sym_d^{\rm op}) \right),\eneq
  where $\Psi_{n,d}$ denotes the action of the  symmetric group $\sym_d$ on $V^{\tens d}$.

Even though it is not the subject studied in the present paper, the referee pointed out that 
 the centralizer algebra on the right hand side in \eqref{eq:dimGn}  is also interesting, since it is related to the restriction of modules from $GL(V)$ to $G$ (see Remark \ref{rem: weyl module}). 

  Now the stability of the Hilbert series for the case of orthogonal groups $O_n(\C)$ can be explained as follows: 

    1) By Schur-Weyl duality for $O_n(\C)$, for sufficiently large $n$, there is an isomorphism from the \emph{Brauer algebra} $\Br_d(n)$  to the algebra  $\End_{O_n(\C)}(V^{\tens d})$ of $O_n(\C)$-module endomorphism, which is an extension of $\Psi_{n,d}$.
  
   2) The dimensions of the centralizers $\Cent_{\Br_d(\delta)}(\cor\sym_d)$ are the same for all parameters $\delta$.
    
 Replacing 1) by Schur-Weyl duality for $Sp_{2m}(\C)$, if $m$ is sufficiently large, then 
there is an isomorphism of algebras from $\Br_d(-2m)$ to $\End_{Sp_{2m}(\C)}(V^{\tens d})$.
Hence we also obtain the stability of the Hilbert series for the symplectic groups $Sp_{2m}(\C)$. 
Moreover, we have  $\dim_\C \Cent_{\Br_d(n)}(\cor\sym_d) = \dim_\C \Cent_{\Br_d(-2m)}(\cor\sym_d) $ for all  $n,m \in \Z_{\ge 1}$ by 2),  we obtain the equality between the stable Hilbert series for $O_n(\C)$ and that of $Sp_{2m}(\C)$.

 Note that   the above 1) and 2) hold with a large class of fields.
In \cite{DH}, Doty and Hu proved Schur-Weyl duality for the orthogonal group $O(n, q')$ for infinite fields $\cor$ with $\ch(\cor)\neq 2$ and in \cite{DDH}, Dipper, Doty and Hu proved Schur-Weyl duality for the symplectic group $Sp_n(\cor)$ for arbitrary infinite fields.
Thus the equality \eqref{eq:dimGn} concludes the stability of the Hilbert series for those cases, too.
 Note that the orthgonal groups $O(n,q')$ in \cite{DH}  are  different from  but conjugate to $O_n(\C)$ which Willenbring had given in \cite{Willenbring} (see subsection \ref{og}). 
Summarizing, for sufficiently large $n$, we have the following equalities
$$ \dim_\cor P^d(M_n(\cor))^{O(n, q')}=\dim_\cor P^d(M_n(\cor))^{Sp_n(\cor)}=\dim_\cor \Cent_{\Br_d(\delta)}(\cor \sym_d)=\# \ULG_d^O $$
under some conditions on $\cor$ (Corollary \ref{cor:orthsymp}). 

Let $\widetilde \beta_d$ be the set of \emph{Brauer diagrams} in $\Br_d(\delta)$. 
Then the centralizer $\Cent_{\Br_d(\delta)}(\cor \sym_d)$ has a natural basis consisting of the orbit sums of the $\sym_d$ conjugation action on $\widetilde \beta_d$.
Note that there is a bijection between $\widetilde \beta_d$ and the set $\widetilde I_{2d}$ of fixed-point free involutions on $\{1,2,\ldots,2d\}$. 
Interesting enough, during the identification of $\dim_\cor P^d(M_n(\cor))^{O(n, q)}$ with the number of directed graphs  in  $\ULG_d^O$, Willenbring considered an action of $\sym_d$ on the set $\widetilde I_{2d}$ and it turned out that there is a bijection from the set of orbits to the set $\ULG_d^O$ (\cite[Section 3]{WillenbringO}). 
Because the action of $\sym_d$ coincides with the conjugation action on $\widetilde \beta$, Willenbring's bijection can be understood as a bijection between the natural basis of $\Cent_{\Br_d(\delta)}(\cor \sym_d)$ and the set $\ULG_d^O$
(See Remark  \ref{rem:ShalileWillenbring}.)
 On the other hand, in  \cite{Shalile2011}, Shalile studied the basis of the centralizer $\Cent_{\Br_d(\delta)}(\cor \sym_d)$. It is parametrized by the \emph{generalized cycle types}, which are multisets of certain equivalence classes of sequences in the letters $\{U,L,T\}$. 
Observing a similarity between the construction of generalized cycle types and Willenbring's bijection, 
we give a concrete bijection between the set $\ULG_d^O$  and the set $\GCT_d$ of generalized cycle types (Section \ref{GCT}).

\medskip
In the perspective of the equality \eqref{eq:dimGn}, one may consider other families $\{G_n\}_{n\ge 1}$ different from the orthogonal groups and the symplectic groups.  
We studied the cases $G_n=\Sigma_n$  where 
$\Sigma_n$ denotes the group  of permutation matrices in $GL_n(\cor)$ 
provided
$\ch(\cor)=0$.
By Schur-Weyl duality for $\Sigma_n$ (\cite{HR,Jones}),  if  $n\ge 2d$, then there is an isomorphism  $\Psi_{n,d}$
from the \emph{partition algebra} $\Par_d(n)$ to  $\End_{\cor \Sigma_n}(V^{\tens d})$, which is an extension of the action of symmetric group $\sym_d$ on $V^{\tens d}$.
Recall that the partition algebra $\Par_d(\delta)$ has the diagram basis $\beta_d$ parametrized by the set $\Pi_{2d}$ of set partitions of $\{1,2,\ldots, 2d\}$.  
Since the symmetric group $\sym_d$ acts on  the set $\beta_d$ by conjugation, the centralizer 
$\Cent_{\Par_d(\delta)}(\cor\sym_d)$ has a natural basis consisting of the orbit sums.
Hence one can deduce the stability of the Hilbert series of the ring $P(M_n(\cor))^{\Sigma_n}$ similarly to the case of the Brauer algebra.

Actually, the partition algebra case is better than the Brauer algebra case in the following sense: there is a nice basis $\set{x_\pi}{\pi \in \Pi_{2d}}$ of $\Par_d(n)$ which is compatible with the image and the kernel of the homomorphism $\Psi_{n,d}$ (\cite{HR, BH1}).
Hence the dimensions in \eqref{eq:dimGn} are calculable even in the case $n < 2d$.
A key result in this direction is to extend the Willenbring's bijection between $\widetilde \beta_d$ and $\ULG_d^O$ in the case of partition algebra.
 We find a bijection between the set $\Pi_{2d}$ of set partitions and the set $\LG_d$ of \emph{multidigraphs with $d$ arrows labeled by $\{1,\ldots,d\}$ bijectively}.
This bijection very  is convenient, because the conjugation action of $\sym_d$ on $\Pi_{2d}$ is just the permuting the edge labels of multidigraphs in $\LG_d$ (Theorem \ref{thm:action}). 
In conclusion, we have a bijection between the natural basis of the centralizer $\Cent_{\Par_d(\delta)}(\cor\sym_d)$ and the set  $\ULG_d$ of (unlabeled) multidigraphs with $d$ arrows.
Then by the compatibility of the basis $\set{x_\pi}{\pi \in \Pi_{2d}}$ with $\Psi _{n,d}$, we show that 
the dimension of  $P^d(M_n(\cor))^{\Sigma_n}$ is equal to the cardinality of $\ULG_{d, \le n}$, where $\ULG_{d, \le n}$ denotes the set of multidigraphs with  $d$ arrows whose number of non-isolated vertices is less than or equal to $n$ (Corollary \ref{cor:main_result_for_n_<_2d}). 
 A similar result appeared in \cite{MW} (see Remark \ref{rem:MW}).

We present the results on the partition algebra at first and the results on the Brauer algebra next, because it is natural to explain the Brauer algebra cases as a restriction of the partition algebra case.
\medskip

{\bf Acknowledgements.}  
The first author thank Jonathan Axtell and  Sangjib Kim for introducing useful references. 
The authors would like to thank the anonymous reviewer for valuable comments and suggestions.

\section{Polynomial invariants on matrices and Schur algebras} 

\subsection{Homogeneous polynomial functions on matrices}
For a finite dimensional vector space $W$ over $\cor$, let $P(W)$ be the subalgebra of the algebra  $\cor^W$  of functions from $W$ to $\cor$ generated by the dual space $W^*$ of $W$. We call  $P(W)$  \emph{the algebra of polynomial functions} on $W$. 
Let  $\{x_i \}_{i\in I}$ be a basis of $W^*$.
Then the algebra $P(W)$ is equal to the subalgebra $\cor[x_i ; i \in I]$ of $\cor^{W}$ generated by the functions $x_i$.
Since $\cor$ is an infinite field,  the algebra $P(W)=\cor[x_i ; i \in I]$ can be identified with the polynomial algebra with $\# I$ indeterminates.

A polynomial function $f \in P(W)$ is called \emph{homogeneous of degree $d$} if $f(cw)=c^d f(w)$ for all $w\in W$ and $c \in \cor$.

We have a vector space decomposition
\eqn \label{eq:decomp}
P(W) = \soplus_{d\ge 0}P^d(W),
\eneqn
where $P^d(W)$ denotes the space of homogeneous polynomials of degree $d$.
Note that  $P^d(-)$ is a contravariant endofunctor on the category of finite-dimensional vector spaces over $\cor$: for  a linear map $\phi:W \to V$,  $P^d(\phi) : P^d(V) \to P^d(W)$ is the linear map  given by
$f\mapsto P^d(\phi)(f):=f\circ \phi$ for $f \in P^d(W)$.

Note that $d ! \neq 0$ in a field $\cor$ if and only if either $\ch (\cor) =0$ or $\ch (\cor) >d$.
The following lemma is well-known.
\begin{lemma} \rm (see, for example, \cite[Section 1.5]{Pirashvili}\rm{)} \label{lem:Kuhn_dual}
Assume that either $\ch (\cor) =0$ or $\ch (\cor) >d$.
Then there is an isomorphism $$\eta: P^d(-) \isoto P^d(-^*)^*$$ between the contravariant endofunctors on  the category of finite-dimensional vector spaces over $\cor$.
\end{lemma}

Our main object is the space  $P^d(M_n(\cor))$  of homogeneous polynomial functions on the space  $M_n(\cor)$  of $n \times n $ matrices. 
There is a left action of the group $GL_n(\cor)$ of invertible $n\times n$ matrices on $P^d(M_n(\cor))$  by conjugation:
\eqn
(g. f)(X) : = f(g^{-1} X g) \qquad (g\in GL_n(\cor), \ f \in P^d(M_n(\cor)), \ X \in M_n(\cor)).
\eneqn
The dual space $P^d(M_n(\cor))^*$  is also a left $GL_n(\cor)$-module with the action given by 
\eqn
(g. \Psi)(f) : = \Psi(g^{-1}.f) \qquad (g\in GL_n(\cor), \ \Psi \in P^d(M_n(\cor))^* \ f\in  P^d(M_n(\cor))).
\eneqn

\begin{prop} \rm \label{prop:dual}
Let $ n,d \in \Z_{\ge 1}$.
Assume that either $\ch (\cor) =0$ or $\ch (\cor) >d$.
Then the $GL_n(\cor)$-module $P^d(M_n(\cor))$ is isomorphic to  its dual $P^d(M_n(\cor))^*$.
\end{prop}
\begin{proof}
Let $\Phi: M_n(\cor)\to M_n(\cor)^*$  be the  linear isomorphism given by
$$ A \mapsto (B \mapsto tr(AB)) \qquad \text{for} \ A,B \in M_n(\cor).$$

Let $c_g : M_n(\cor) \to M_n(\cor) $ be the map given by $X \mapsto gX g^{-1}$.
Then we have a commutative diagram below for every $g\in GL_n(\cor)$:
\eq \label{eq:trace}
\xymatrix{
M_n(\cor)^* \ar[r]^{\Phi^{-1}} \ar[d]_{ (c_{g^{-1}})^* } & M_n(\cor)  \ar[d]_{c_{g}}\\
M_n(\cor)^* \ar[r]^{\Phi^{-1}}  & M_n(\cor).
}
\eneq
The following diagram is commutative
\eqn
\xymatrix@C=8em@R=5em{
P^d(M_n(\cor)) \ar[r]^{\eta_{M_n(\cor)}} \ar[d]_{P^d(c_{g^{-1}})} &P^d( M_n(\cor)^* )^* \ar[r]^{P^d(\Phi^{-1})^*} \ar[d]_{ P^d((c_{g^{-1}})^*)^* } & P^d(M_n(\cor))^*  \ar[d]_{P^d(c_{g})^*}\\
P^d(M_n(\cor)) \ar[r]^{\eta_{M_n(\cor)}} & P^d(M_n(\cor)^*)^* \ar[r]^{P^d(\Phi^{-1})^*}  & P^d(M_n(\cor))^*,
}
\eneqn
where the left square comes from the isomorphism in Lemma  \ref{lem:Kuhn_dual} and the right square is obtained by applying the covariant functor $P^d(-)^*$ to the square in \eqref{eq:trace}.

Since  the actions of $GL_n(\cor)$ on $P^d(M_n(\cor))$  and $P^d(M_n(\cor))^*$ satisfy that $g.f= P^d(  c_{g^{-1}} )(f)$ and $g.\Psi=P^d(c_{g})^*(\Psi)$, respectively,
the composition $$P^d(\Phi^{-1})^* \circ \eta_{M_n(\cor)} : P^d(M_n(\cor)) \to  P^d(M_n(\cor))^*$$ is an isomorphism of $GL_n(\cor)$-modules, as desired.
\end{proof}

\subsection{Schur algebras and polynomial invariants on matrices}

Let $V=\cor^n$ be the natural representation of $GL_n(\cor)$.
Then $GL_n(\cor)$ acts diagonally on the $d$-th tensor power $V^{\tens d}$ from the left.  
For a group $G$, the group algebra over $\cor$ is denoted by $\cor G$. 
Let us denote by 
\eqn 
\Phi_{n,d} : \cor GL_n(\cor) \to \End_\cor(V^{\tens d})
\eneqn
 the algebra homomorphism induced  by the action.

On the other hand the symmetric group $\sym_d$ acts on $V^{\tens d}$ from the right by place permutation:
 \eq \label{eq:right action of S_d}
(w_1 \tens \cdots \tens w_d).\sigma := (w_{\sigma(1)} \tens \cdots \tens w_{\sigma(d)} ) \qquad \text{for} \ w_1,\ldots,w_d \in V, \ \text{and} \ \sigma \in \sym_d.
\eneq
It induces an algebra homomorphism $\Psi_{n,d}  : \cor \sym_d^\opp \to  \End(V^{\tens d})$,  where $A^\opp$ denotes the opposite ring of a ring $A$.
Then we have 
\eq \Psi_{n,d}(\cor\sym_d^\opp) = \End_{\cor GL_n(\cor)}(V^{\tens d}). \label{eq:SWsymm1}
\eneq
If $n \ge d$, then $\Psi_{n,d}$ is an isomorphism of $\cor$-algebras.

We call the image $\Phi_{n,d} ( \cor GL_n(\cor))$  the \emph{Schur algebra} and denote it  $S_{n,d}(\cor)$.
Then the centralizer of the symmetric group action on $V^{\tens d}$ is the same as the Schur algebra $S_{n,d}(\cor)$  (for example, see \cite[Theorem 2.13]{Martin}):
\eqn
S_{n,d}(\cor):=\Phi_{n,d} ( \cor GL_n(\cor))=\End_{\cor \sym_d^{\opp} } (V^{\tens d}).
\eneqn

For  each function $f$ on $GL_n(\cor)$, define a function $\tilde f$ on $\cor GL_n(\cor)$ as follows: if $\phi = \sum_{g \in GL_n(\cor)} a_g g$, then $\tilde f(\phi):=\sum_{g \in GL_n(\cor)} a_g f(g)$.

An important relation between $P^d(M_n(\cor))^*$ and $S_{n,d}(\cor)$ can be summarized as in the following theorem. 
\begin{theorem} \rm(\cite[Theorem 2.2.4]{Martin}\rm) \
There is a nondegenerate bilinear form $$\lan \ , \ \ran : P^d(M_n(\cor)) \times S_{n,d}(\cor) \to \cor$$ given by
\eqn
\lan f, \xi \ran := \tilde f(\phi) ,\quad \text{where} \ f \in  P^d(M_n(\cor)), \quad \xi=\Phi_{n,d}(\phi) \in S_{n,d}(\cor) \ \text{for some} \ \phi \in \cor GL_n(\cor). 
\eneqn
\end{theorem}

Define  a left $GL_n(\cor)$-module structure on $S_{n,d}(\cor)$ by
\eqn
\lan f, g.\xi \ran  : =\lan  g^{-1}.f, \xi \ran   \qquad \text{for} \ f \in  P^d(M_n(\cor)), \ \xi \in S_{n,d}(\cor), \ \text{and} \ g \in G.
\eneqn
Thus we have an isomorphism between  $P^d(M_n(\cor))^*$ and $S_{n,d}(\cor)$ as $GL_n(\cor)$-modules given by $\xi \mapsto \langle-, \xi \rangle$ for all $\xi \in S_{n, d}(\cor)$.

\begin{prop} \rm \label{prop:action}
We have 
\eqn
g.\xi = \Phi_{n,d}(g) \circ \xi \circ  \Phi_{n,d}  (g^{-1}) \qquad \text{for} \ g \in GL_n(\cor), \ \xi \in S_{n,d}(\cor).
\eneqn
Hence  for any subgroup $G$ of $GL_n(\cor)$, we have 
$$S_{n,d}(\cor)^G=\Cent_{\End_{\Phi_{n,d}(\cor G)}(V^{\tens d})}\left( \Psi_{n,d}(\cor \sym_d^{\rm op}) \right).$$
\end{prop}
\begin{proof}
 By the linearity we may assume that $\xi =\Phi_{n,d}(\phi)$ for some $\phi \in GL_n(\cor)$. 
Then we have 
\eqn
\lan f, g.\xi \ran  =\lan  g^{-1}.f, \xi \ran  =\widetilde{(g^{-1}.f)}( \phi)
=(g^{-1}.f)( \phi) = f(g \phi g^{-1}) = \lan f, \Phi_{n,d}(g \phi g^{-1}) \ran
\eneqn
so that 
\eqn 
g.\xi  = \Phi_{n,d}(g) \circ \Phi_{n,d}(\phi) \circ \Phi_{n,d}( g^{-1})=  \Phi_{n,d}(g) \circ \xi \circ \Phi_{n,d}( g^{-1}),
\eneqn
as desired.
\end{proof}

Combining Proposition \ref{prop:dual} and Proposition \ref{prop:action},  we obtain the following theorem.
\begin{theorem} \rm \label{thm:dim}
Let $\cor$ be an infinite field.
Assume that  $\ch (\cor)=0$ or $\ch (\cor) > d$. Then for any subgroup $G$ of $GL_n(\cor)$, we have 
$$\dim_\cor P^d(M_n(\cor))^G  = \dim_\cor \Cent_{\End_{\Phi_{n,d}(\cor G)}(V^{\tens d})}\left( \Psi_{n,d}(\cor \sym_d^{\rm op}) \right).$$
\end{theorem}

\begin{example}
Let $G=GL_n(\cor)$ in Theorem \ref{thm:dim}.
 Assume that $\ch(\cor)=0$. 
It is well-known that  the algebra $P(M_n(\cor))^{GL_n(\cor)}$ is isomorphic to the $\cor$-algebra of symmetric polynomials in $n$ variables as graded algebras (see, for example, \cite[Example 1.2]{Dolg}). It follows that
\eqn 
\dim_\cor (P^d(M_n(\cor))^{GL_n(\cor)})  = \#\set{\la \vdash d}{ \ell(\la) \le n},
\eneqn
 where $\la \vdash d$ denotes a partition $\la$ of $d$, and $\ell(\la)$ denotes the length of a partition $\la$.

On the other hand,  
the dimension of $P^d(M_n(\cor))^{GL_n(\cor)}$ is equal to the dimension of the center of the algebra  $\Psi_{n,d} (\cor \sym_d^\opp) $ by Theorem \ref{thm:dim} and \eqref{eq:SWsymm1}.
Thus one can recover the above equality  in this case from a result in  \cite[Theorem 7]{GePr}:
 the set $\set{\Psi_{n,d}(c_\la)}{\la \vdash d, \ \ell(\la) \le n }$ forms a basis of the center of $\Psi_{n,d} (\cor \sym_d^\opp) $, where 
$c_\la$ denotes the sum of all permutations whose cycle type is $\la$ for a partition $\la$ of $d$.

\end{example}

\begin{remark} \label{rem: weyl module}
Even though we are focusing on a relation of the algebra 
$S_{n,d}(\cor)^G$ 
with the ring of invariants $P(M_n(\cor))^G$ in the present paper, the algebra $S_{n,d}(\cor)^G$ is also interesting in the following reason: 
If $G=O_n(\C), Sp_n(\C)$ or a finite subgroup of $GL_n(\C)$, and  $M$ is a finite-dimensional representation of $\C \sym_d^\opp$, then the algebra $S_{n,d}(\C)^G$ acts on the space
$\check M:=  \Hom_{\C\sym_d^\opp}(M,V^{\tensor d})$ in a natural way.
Moreover the action commutes with the one of $G$ on $\check M$.  
Hence the multiplicity spaces for the restriction of $\check M$ from $GL_n(\C)$ to $G$ are  representations of the algebra $S_{n,d}(\C)^G$.

\end{remark}

\vskip1em

\section{Group of permutation matrices and  partition algebras}

\subsection{Partition algebra and Schur-Weyl duality}
We recall the partition algebra and its bases following \cite[Section 2.1,  Section 2.2]{BH1}. 
\vskip 1em

For $d \in \mathbb{Z}_{\ge 1}$, we consider the set partitions of $[1, 2d]:= \{1, 2, \cdots , 2d\}$ into disjoint nonempty subsets, called $\emph{blocks}$, and define 
$$\Pi_{2d}:=\{\text{set partitions of $[1, 2d]$} \}.$$

For each $\pi \in \Pi_{2d}$ let $|\pi|$ be the number of blocks of $\pi$.  
 We associate a  \emph{diagram $D_\pi$} to $\pi$ as follows:
It has two rows of $d$ vertices each, with the bottom vertices indexed by $1, 2, \dots , d$ and the top vertices indexed by $d+1, d+2, \cdots , 2d$ from left to right. 
Vertices are connected by an edge if they lie in the same block of $\pi$. 
Note that the way the edges are drawn is immaterial, what matters is that the connected components of the diagram $D_\pi$ correspond to the blocks of the set partition $\pi$. 
Thus, $D_\pi$ represents the equivalence class of all diagrams with connected components equal to the blocks of $\pi$. 

We write briefly
\eqn i'\seteq
\begin{cases}
i+d & \text{if} \ 1 \le i \le d, \\
i-d & \text{if} \  d+1\le i\le 2d.
\end{cases}
\eneqn
We denote $\beta_d$ as the set $\{D_\pi | \pi \in \Pi_{2d} \}$. 
Sometimes we will confuse $\pi$ and $D_\pi$ for simplicity.

\vskip 1em
\begin{definition} \rm
Let $\cor$ be a field and $\delta \in \cor$. 
The \emph{partition algebra} $\Par_d(\delta)$  is an associative $\cor$-algebra with $\beta_d$ as a basis where the multiplication $D_{\pi_1} D_{\pi_2}$ for any two diagrams $D_{\pi_1}, D_{\pi_2}$ is defined as the following: 
\begin{enumerate}[(i)]
\item Draw two diagrams vertically, $D_{\pi_1}$ on the top and $D_{\pi_2}$ at the bottom.
\item Identify the vertices in the bottom row of $D_{\pi_1}$ with those in the top row of $D_{\pi_2}$. 
\item Delete all connected components that entirely lie in the middle row of the joined diagrams. 
 We denote by $\pi_1 \ast \pi_2$  the set partition represented by the thusly obtained diagram. 
\item  The product is given by $D_{\pi_1} D_{\pi_2}=\delta^{[\pi_1 \ast \pi_2]}D_{\pi_1 \ast \pi_2}$,
where $[\pi_1 \ast \pi_2]$ denote the number of blocks removed from the middle row.
\end{enumerate}
We call $\beta_d$ the \emph{diagram basis} of $\Par_d(\delta)$. 
\end{definition}
The set partition $\pi_1\ast \pi_2$ and the nonnegative integer $[\pi_1\ast \pi_2]$  depend only on the underlying set partitions $\pi_1, \pi_2$ and are independent of the diagrams chosen to represent them. In particular, the product $D_{\pi_1} D_{\pi_2}$ depends only on the set partitions $\pi_1$ and $\pi_2$.

Note that  a set partition of $[1,2d]$ each of whose blocks is of the form $\{i,j'\}$ for some $1\le i,j \le d$ can be identified with the permutation $\sigma$ in $\sym_d$ given by $\sigma(i)=j$. 
Under this correspondence,  the group algebra $\cor \sym_d$ is embedded into $\Par_d(\delta)$.
Note that we have
\eq \label{eq:action_on_diagram}
D_\sigma D_{\pi} D_{\sigma'} = D_{\sigma * \pi  *\sigma'}
\eneq
for $\pi \in \Pi_{2d}$ and $\sigma, \sigma'\in \sym_{d}$.
\medskip
In particular, the symmetric group $\sym_d$ acts on $\Pi_{2d}$ by conjugation:
 $$\sigma.\pi : = \sigma \ast \pi \ast \sigma^{-1}.$$

Let $\cor$ be a field  
and 
$V$ be a finite-dimensional $\cor$-vector space. We fix a basis $\{ v_i \in V \, | \,  i=1, 2, \dots , n \}$ and  identify $GL_n(\cor)$ with the group of automorphisms on $V$.
 Let  $\Sigma_n$ be the subgroup of $GL_n(\cor)$ consisting of all the permutation matrices. 
Then the group $\Sigma_n$ is isomorphic to the symmetric group $\sym_n$.

For each sequence ${\bold i}=(i_1,\ldots,i_d) \in [1,n]^d$, set 
$$v_{\bold i} :=v_{i_1} \tens v_{i_2} \tens \cdots \tens v_{i_d}$$
so that $\set{v_{\bold i}}{\bold i \in [1,n]^d}$ forms a basis of the tensor power $V^{\tens d}$.

For ${\bold s}, \bold{r} \in [1, n]^d$,
let $E_{\bold s} ^{\bold r}$ be the element in $\End(V^{\tens d})$ given by $$E_{\bold s} ^{\bold r}v_{\bold t}=\delta_{\bold s, \bold t}v_{\bold r}$$
 for all $\bold t \in [1,n]^d$.
For  each $\pi \in \Pi_{2d}$,  define
$$N_d(\pi)=\{(u_1,\ldots,u_{2d} ) \in [1,n]^{2d} \ \text{such that if $i$ and $j$ belong to the same block of $\pi$ then $u_i=u_j$}\}.$$

Define a linear map $\Psi_{n, d} : \Par_d(n) ^\opp \longrightarrow \End(V^{\tens d})$ by 
$$\Psi_{n, d}(D_\pi)=\displaystyle \sum_{(\bold r, \bold r') \in N_d(\pi)} E_{\bold r'} ^{\bold r}.$$

\vskip 1em
\begin{theorem} \rm (\cite[Theorem 3.6]{HR}, \cite{Jones}) \rm 
\label{swdsg}
Assume that   $\cor$ is a  field with  $\ch (\cor)=0$.
The map $\Psi_{n, d}$ is an algebra homomorphism whose image 
$\Psi_{n, d}(\Par_d(n)^\opp)$ is equal to $\End_{\cor \Sigma_n (\cor)} (V^{\tens d})$. 
Moreover if $n \ge 2d$, then $\Psi_{n,d}$ is injective. 
\end{theorem}

Note that 
the restriction of $\Psi_{n,d}$ onto ${\sym_d^\opp}$  is equal to \eqref{eq:right action of S_d}. 

Hence combining the Theorem \ref{swdsg}  with Theorem \ref{thm:dim}, we obtain
\begin{theorem} \rm
Assume that   $\cor$ is a  field with  $\ch (\cor)=0$.

If $n \ge 2d$, then
$$\dim_{\cor}P^d(M_n(\cor))^{\Sigma_n}=\dim_{\cor}\Cent_{\Par_d(n)}(\cor\sym_d).
$$
\end{theorem}

\subsection{Multidigraphs and the Centralizer of $\sym_d$ in $\Par_d(\delta)$}
\vskip 1em

\begin{definition} \rm
A \emph{directed multigraph}, shortly a \emph{multidigraph}, is an unlabeled graph  that is made of   countably many unlabeled  vertices and a set of edges  with  directions, called \emph{arrows}. 
We call a multidigraph is vertex-labeled (respectively, edge-labeled) if there is a function from the set of vertices (respectively, the set of edges) to a set of labels. 
\end{definition}

We  will associate  three  multidigraphs $\widetilde \psi_d(\pi), \psi_d(\pi)$ and $\phi_d(\pi)$ to a set partition $\pi \in \Pi_{2d}$ as follows: Let $\widetilde \psi_d(\pi)$ be the vertex-labeled and edge-labeled multidigraph  where (1) the vertices are labeled by the blocks of $\pi$, and (2) the multidigraph has $d$ arrows such that  for each $1\le i \le d$, there is exactly one arrow with the label $i$ which starts from the vertex containing $i$ and ends at the vertex containing $i'$.
Let $\psi_d(\pi)$ be the edge-labeled multidigraph obtained by  removing the labels of the vertices of $\widetilde \psi_d(\pi)$. Finally, removing the edge labels of $\psi_d(\pi)$, we obtain a  unlabeled multidigraph $\phi_d(\pi)$.

We define $\LG_d$ the set of edge-labeled mutidigraphs of $d$ arrows which do not contain isolated vertices and are labeled by  the set $\{1, 2, \dots , d\}$ bijectively.  
Also, let $\ULG_d$ be the set of multidigraphs with $d$ arrows which do not contain isolated vertices.

Then,  the map $\psi_d:\Pi_{2d} \to\LG_d$ is bijective and its inverse is given as follows:
\begin{enumerate}[(i)]
\item For a given  multidigraph in $\LG_d$, label all the vertices with the empty sets.  
\item For each $1\le i \le d$, add $i$ to the source  and $i'$ to the target of the arrow with the label $i$. 
The label of a vertex of the graph forms a block of $\psi_d^{-1}(g)$.
\end{enumerate}

\vskip 1em
\begin{theorem} \rm
\label{thm:action}
For $\sigma \in \sym_d$ and $\pi \in \Pi_{2d}$,
the multidigraph  $\psi_d(\sigma \ast \pi \ast \sigma^{-1} )$ is obtained from $\psi_d(\pi )$ by permuting the arrow labels by $\sigma$.
\end{theorem}
\begin{proof}
 We may assume that $\sigma$  is a simple transposition $ (i ~ i+1)$ for some $1\le i\le d-1$. 
Let $b_1, b'_1, b_2$ and $b'_2$ be the blocks of $\pi$ containing $i,  i', i+1$ and $(i+1)'$ respectively. 
Note that the set partition $ \sigma \ast \pi$ is obtained from $\pi$ 
by changing the blocks $b_1$ and $b_2$ into $c_1:=(b_1\setminus \{i\})\cup \{i+1\}$ and 
$c_2:=(b_2\setminus \{i+1\})\cup \{i\}$, respectively.  Similarly the set partition 
 $ \pi \ast \sigma^{-1}$ is obtained from $\pi$ 
by changing the blocks $b'_1$ and $b'_2$ into $c'_1:=(b'_1\setminus \{i'\})\cup \{(i+1)'\}$ and 
$c'_2:=(b_2\setminus \{(i+1)'\})\cup \{i'\}$, respectively.
Thus the set partition $\sigma*\pi*\sigma^{-1}$ is obtained from $\pi$ 
by changing the blocks $b_1, b'_1, b_2$ and $b'_2$ into the subsets $c_1,c'_1,c_2$ and $c'_2$, respectively, and all the other blocks remain the same.
It follows that the multidigraph $\psi_d(\sigma*\pi*\sigma^{-1})$ is obtained from $\psi_d(\pi)$ by exchanging the arrows with labels $i$ and $i+1$, as desired.
\end{proof}

The following is an immediate consequence of Theorem \ref{thm:action}.

\vskip 1em
\begin{corollary} \rm
\label{parbasis}
Two set partitions $\pi_1, \pi_2$ are $\sym_d$-conjugate if and only if $\phi_d(\pi_1)=\phi_d(\pi_2)$.
\end{corollary}

For a multidigraph $G \in \ULG_d$, set   
$$E_G:=\{ \pi \in \Pi_{2d} \, | \, \phi_d(\pi)=G\}, \quad \text{and} \quad \gamma_G:=\sum \limits_{\pi \in E_G} D_\pi.$$ That is,  each $E_G$ is an orbit in $\Pi_{2d}$ of the conjugation action of $\sym_d$, and $\gamma_G$ is the sum of the diagram basis elements  in the orbit $E_G$.

\begin{corollary} \rm \label{cor:dim=num.of.mutlidigraph}
The set $$\Gamma_d:=\{ \gamma_G  \in \Par_d(\delta) \, | \, G \in \ULG_d\}$$  
forms  a basis for $\Cent_{\Par_d(\delta)}(\cor\sym_d)$, the centralizer of $\cor\sym_d$ in $\Par_d(\delta)$.  
In particular,  the dimension of $\Cent_{\Par_d(\delta)}(\cor\sym_d)$ is equal to $\#  \, \ULG_d$, the number of unlabeled multidigraphs with $d$ arrows and no isolated vertices. The dimension $\dim_\cor \Cent_{\Par_d(\delta)}(\cor\sym_d)$ is independent from the choice of $\delta$.
\end{corollary}

\begin{proof}
Since $\{D_\pi\}_{\pi \in \Pi_{2d}}$ is a basis for $\Par_d(\delta)$, the set $\Gamma_d$ is linearly independent over $\cor$ by definition. 

Write an element $v$ in $\Par_d(\delta)$ as $v=\sum_{\pi \in \Pi_{2d}}a_\pi D_\pi$ for some $a_\pi \in \cor$.
For any element $\sigma \in \sym_{d}$, we have 
\eqn
D_\sigma v D_{\sigma^{-1}} = \sum_{\pi \in \Pi_{2d}}a_\pi D_\sigma D_\pi D_{\sigma^{-1}} 
= \sum_{\pi \in \Pi_{2d}} a_\pi  D_{\sigma \ast  \pi \ast  \sigma^{-1}}
= \sum_{\pi \in \Pi_{2d}} a_{\sigma^{-1}  \ast \pi  \ast \sigma}  D_{\pi}.
\eneqn
Hence  $v=D_\sigma v D_{\sigma^{-1}} $ if and only if $a_{\sigma^{-1}  \ast \pi  \ast \sigma} = a_\pi$ for all $\pi \in \Pi_{2d}$.
Thus for $v \in \Cent_{\Par_d(\delta)}(\cor\sym_d)$, we have
\eqn
v=\sum_{G\in \ULG_{d} }a_{G} \  \gamma_G
\eneqn
for some $a_G \in \cor$, 
as desired.
\end{proof}

\begin{corollary} \rm \label{cor:main_result_for_n_ge_2d}
Let $\cor$ be an infinite field. If $n \ge 2d$, 
 then
the dimension of the space $P^d(M_n(\cor))^{\Sigma_n}$ is equal to $\# \, \ULG_{d}$, the number of multidigraphs with  $d$ arrows and  no isolated vertices.
In particular,  the sequence $\left( \dim_{\cor}P^d(M_n(\cor))^{\Sigma_n} \right)_{n =1}^{\infty}$ is  stable for $n \ge 2d$.
\end{corollary}
\begin{proof}
If $\ch(\cor) =0$, then we have
\eqn \dim_\cor P^d(M_n(\cor))^{\Sigma_n} = 
\dim_\cor \Cent_{\End_{\Phi_{n,d}(\cor \Sigma_n)}(V^{\tens d})}\left( \Psi_{n,d}(\cor \sym_d^{\rm op}) \right) = \dim_\cor \Cent_{\Par_d(n)} \left( \cor \sym_d\right)
\eneqn
by Theorem \ref{thm:dim} and Theorem \ref{swdsg}. Hence we get the desired result by Corollary
\ref{cor:dim=num.of.mutlidigraph}.
Since $P(M_n(\cor))^{\Sigma_n}$ is the invariant ring associated with the  permutation representation induced from the conjugation action of $\Sigma_n$ on the standard basis of $M_n(\cor)$, its Hilbert series is independent of the ground field $\cor$ (see, for example  \cite[Corollary 3.2.2]{NS}). 
\end{proof}

\vskip 1em

\subsection{The Orbit basis and the Centralizer of $\sym_d$ in $\Par_d(\delta)$ when $n<2d$}
In this subsection, we recall the \emph{orbit basis}  of $\Par_d(\delta)$ following \cite[Section 2.3]{BH1}.
The set $\Pi_{2d}$ of set partitions of $[1, 2d]$ forms a lattice  under the partial order given by 
\begin{center}
$\pi \preceq \rho$ if every block of $\pi$ is contained in a block of $\rho$. 
\end{center}
In this case we say that $\pi$ is a \emph{refinement of $\rho$} and that $\rho$ is a \emph{coarsening of $\pi$}. 

 The \emph{orbit basis}  $\{ x_\pi  \, | \, \pi \in \Pi_{2d} \}$ of $\Par_d(\delta)$ is defined by the following coarsening relation with respect to the diagram basis: 
\eqn D_\pi :=\displaystyle \sum_{\pi \preceq \rho} x_\rho.
\eneqn 
Then
\eq \label{eq:expan_orbit}
x_\pi =\displaystyle \sum_{\pi \preceq \rho} \mu_{2d}(\pi,\rho) D_\rho
\eneq
for some integers $\mu_{2d}(\pi,\rho)$ each of which
 satisfies that 
\eq \label{eq:action_on_mobius}
\mu_{2d}(\pi,\rho) = \mu_{2d}(\sigma*\pi*\sigma',\sigma*\rho*\sigma') \qquad \text{for }  \ \sigma, \sigma' \in \sym_d.
\eneq
See, for example, an explicit formula for $\mu_{2d}(\pi, \rho)$ in \cite[(2.18)]{BH1}.

It follows by \eqref{eq:action_on_diagram},  \eqref{eq:expan_orbit}  and \eqref{eq:action_on_mobius} that 
\begin{equation} \label{eq:action_on_orbit}
D_\sigma x_{\pi} D_{\sigma'} = x_{\sigma * \pi  *\sigma'}.
\end{equation}
for $\sigma, \sigma' \in \sym_{d}$ and  $\pi \in\Pi_{d}$ (\cite[Section 4.1]{BH1}).

The orbit basis is interesting, since it is  compatible with the homomorphism $\Psi_{n,d}$ in the following sense.

\begin{theorem} \label{n<2d} \rm (\cite[Theorem 3.8(a)]{BH1}, \cite[Theorem 3.6]{HR}) 
\begin{enumerate}[(i)]
\item The set $\set{ \Psi_{n, d}(x_\pi) }{ \pi \in \Pi_{2d}, \ |\pi| \le n }$ forms a basis of $\Im (\Psi_{n, d})=\End_{\cor\sym_d}(V^{\tens d})$. 
\item The set $\set{ x_\pi }{ \pi \in \Pi_{2d}, |\pi| > n \ }$ forms a basis of  $\Ker (\Psi_{n, d})$.
\end{enumerate}
\end{theorem}

Note that the set of multidigraphs with $d$ arrows and no isolated vertices is  partitioned into
$$\ULG_{d} = \ULG_{d,1} \sqcup \cdots \sqcup \ULG_{d,2d}, $$
where $\ULG_{d,k}$ is the subset consisting of the multidigraphs with $d$ arrows and $k$ non-isolated vertices.
For each $1\le k \le 2d$,  let 
$$\ULG_{d,\le k} := \bigsqcup_{t\le k} \ULG_{d, t}.$$

\begin{theorem} \rm \label{thm:main_partition}
The set $\set{ \sum_{\pi \in E_G} x_\pi }{G\in \ULG_{d}} $ forms a basis of the centralizer 
$\Cent_{\Par_d(\delta)}(\cor \sym_d)$ and 
the set 
$\set{\Psi_{n,d}\left(\sum_{\pi \in E_G} x_\pi \right)}{G\in \ULG_{d,\le n}} $  forms a basis of the centralizer  $\Cent_{\Psi_{n,d}( \Par_d(n)^{\opp} )} \left(\Psi_{n,d}(\cor \sym_d^{\opp} )\right)$. Moreover,
we have
$$\Psi_{n,d}\left(\Cent_{\Par_d(n)^{\opp}}(\cor \sym_d^{\opp})\right) =\Cent_{\Psi_{n,d}( \Par_d(n)^{\opp})} \left(\Psi_{n,d}(\cor \sym_d^{\opp})\right).$$

\end{theorem}
\begin{proof}
Note that $\set{\Psi_{n,d}\left(\sum_{\pi \in E_G} x_\pi \right)}{G\in \ULG_{d,\le t}} $  is linearly independent over $\cor$ for any $1\le t \le 2d$ by Theorem \ref{n<2d} (i).
Write an element $v$ in $\Par_d(\delta)$ as $v=\sum_{\pi \in \Pi_{2d}}a_\pi x_\pi$ for some $a_\pi \in \cor$.
For any element $\sigma \in \sym_{d}$, we have 
\eqn
D_\sigma v D_{\sigma^{-1}} = \sum_{\pi \in \Pi_{2d}}a_\pi D_\sigma x_\pi D_{\sigma^{-1}} 
= \sum_{\pi \in \Pi_{2d}} a_\pi  x_{\sigma \ast  \pi \ast  \sigma^{-1}}
= \sum_{\pi \in \Pi_{2d}} a_{\sigma^{-1}  \ast \pi  \ast \sigma}  x_{\pi}.
\eneqn
Taking $\delta=n$ and applying $\Psi_{n,d}$, we obtain
\eqn
\Psi_{n,d}(D_{\sigma^{-1}}) \Psi_{n,d}(v) \Psi_{n,d}(D_{\sigma}) 
= \sum_{\pi \in \Pi_{2d}} a_{\sigma^{-1}  \ast \pi  \ast \sigma}  \Psi_{n,d}(x_{\pi})
=\sum_{\pi \in \Pi_{2d}, \ | \pi | \le n} a_{\sigma^{-1}  \ast \pi  \ast \sigma}  \Psi_{n,d}(x_{\pi}).
\eneqn
Because of Theorem \ref{n<2d} (i), we conclude that  $\Psi_{n,d}(v)=\Psi_{n,d}(D_{\sigma^{-1}}) \Psi_{n,d}(v) \Psi_{n,d}(D_{\sigma}) $ if and only if $a_{\sigma^{-1}  \ast \pi  \ast \sigma} = a_\pi$ for all $\pi \in \Pi_{2d}$ with $|\pi|\le n$.

Thus for $\Psi_{n,d} (v) \in \Cent_{\Psi_{n,d}(\Par_d(\delta)^{\opp})}(\Psi_{n,d}(\cor\sym_d^{\opp}))$, we have
\eq \label{eq:orbit_orbit_sum}
\Psi_{n,d}(v)=\sum_{G\in \ULG_{d, \le n} }a_{G} \Psi_{n,d}\left(\sum_{\pi \in E_G} x_\pi \right)
\eneq
for some $a_G \in \cor$. 
Hence we obtain the second assertion. The first can be shown in a similar way.
Because
$$\Psi_{n,d}\left(\Cent_{\Par_d(n)^{\opp}}(\cor \sym_d^{\opp})\right) \subset \Cent_{\Psi_{n,d}( \Par_d(n)^{\opp})} \left(\Psi_{n,d}(\cor \sym_d^{\opp})\right) $$
the equation \eqref{eq:orbit_orbit_sum} proves the last assertion, too.
\end{proof}

By the same reasoning as in Corollary \ref{cor:main_result_for_n_ge_2d}, we obtain
\begin{corollary}  \label{cor:main_result_for_n_<_2d}  \rm
Let $\cor$ be an infinite field.
The dimension of the space $P^d(M_n(\cor))^{\Sigma_n}$ is equal to $\# \, \ULG_{d, \le n}$, the number of multidigraphs with  $d$ arrows and no isolated vertices whose number of vertices is less than or equal to $n$.
\end{corollary}

\begin{remark} \label{rem:MW}
The above result appeared in \cite{MW} in the following way.
 In \cite[Theorem 3]{MW}, the number of non-isomorphic multigraphs (without loops) with $d$ edges and $n$ vertices is identified with the dimension of the space $SF(n,d)$ of certain polynomial invariants of degree $d$.
The authors also remarked that the above corollary can be obtained by the same way. See \cite[4.Concluding remarks 2]{MW}.
\end{remark}

\section{Orthogonal groups, symplectic groups and Brauer algebras}

\subsection{Brauer algebras}
Let $\cor$ be a field and $\delta \in \cor$.
For $d \in \mathbb{Z}_{\ge 1}$, set 
\eqn
\widetilde \Pi_{2d}:=\set{\pi\in \Pi_{2d}}{\text{each block of $\pi$ is of size $2$}} \quad  \text{and} \quad
\widetilde \beta_d:=\{ D_\pi | \pi \in \widetilde \Pi_{2d} \} \subset \beta_d.
\eneqn
Then the subspace $B_d(\delta)$ of the partition algebra $\Par_d(\delta)$ spanned by  $\widetilde \beta_d$  is stable under the multiplication. 
We call it the \emph{Brauer algebra} with parameter $\delta$.

Note that the subset  $\{ s_j, e_j \ | \  j=1, 2, \dots , d-1 \}$ of $\widetilde \beta_d$ is a generating set of the algebra $\Br_d(\delta)$, where 
\begin{center}
${\beginpicture 
\setcoordinatesystem units <1.35cm,0.4cm>
\setplotarea x from 0 to 7, y from -1 to 4
\put{$s_j:= $} at 0 1.5
\put{$\bullet$} at  1 0
\put{$\bullet$} at  1 3
\put{$\bullet$} at  1.5 0
\put{$\bullet$} at  1.5 3
\put{$\cdots$} at  2 1.5
\put{$\bullet$} at  2.5 0
\put{$\bullet$} at  2.5 3
\put{$\bullet$} at  3.5 0
\put{$\bullet$} at  3.5 3
\put{$\bullet$} at  4.5 0
\put{$\bullet$} at  4.5 3
\put{$\bullet$} at  5.5 0
\put{$\bullet$} at  5.5 3
\put{$\cdots$} at  6 1.5
\put{$\bullet$} at  6.5 0
\put{$\bullet$} at  6.5 3
\put{$\bullet$} at  7 0
\put{$\bullet$} at  7 3
\put{$1$} at 1 4
\put{$2$} at 1.5 4
\put{$j-1$} at 2.5 4
\put{$j$} at 3.5 4
\put{$j+1$} at 4.5 4
\put{$j+2$} at 5.5 4
\put{$d-1$} at 6.5 4
\put{$d$} at 7 4

\plot 1 3 1 0 /
\plot 1.5 3 1.5 0 /
\plot 2.5 3 2.5 0 /
\plot 3.5 3 4.5 0 /
\plot 3.5 0 4.5 3 /
\plot 5.5 0 5.5 3 /
\plot 6.5 0 6.5 3 /
\plot 7 0 7 3 /
\endpicture}$,
${\beginpicture 
\setcoordinatesystem units <1.35cm,0.4cm>
\setplotarea x from 0 to 7, y from -1 to 4
\put{$e_j:= $} at 0 1.5
\put{$\bullet$} at  1 0
\put{$\bullet$} at  1 3
\put{$\bullet$} at  1.5 0
\put{$\bullet$} at  1.5 3
\put{$\cdots$} at  2 1.5
\put{$\bullet$} at  2.5 0
\put{$\bullet$} at  2.5 3
\put{$\bullet$} at  3.5 0
\put{$\bullet$} at  3.5 3
\put{$\bullet$} at  4.5 0
\put{$\bullet$} at  4.5 3
\put{$\bullet$} at  5.5 0
\put{$\bullet$} at  5.5 3
\put{$\cdots$} at  6 1.5
\put{$\bullet$} at  6.5 0
\put{$\bullet$} at  6.5 3
\put{$\bullet$} at  7 0
\put{$\bullet$} at  7 3
\put{$1$} at 1 4
\put{$2$} at 1.5 4
\put{$j-1$} at 2.5 4
\put{$j$} at 3.5 4
\put{$j+1$} at 4.5 4
\put{$j+2$} at 5.5 4
\put{$d-1$} at 6.5 4
\put{$d$} at 7 4

\plot 1 3 1 0 /
\plot 1.5 3 1.5 0 /
\plot 2.5 3 2.5 0 /
\plot 5.5 0 5.5 3 /
\plot 6.5 0 6.5 3 /
\plot 7 0 7 3 /

\setquadratic
\plot 3.5 3 4 2.5 4.5 3 /
\plot 3.5 0 4 0.5 4.5 0 /
\endpicture}$.
\end{center}

We have a tower of algebras
\eqn
\cor \sym_d \subset \Br_d(\delta) \subset \Par_d(\delta).
\eneqn
In particular, the set $\widetilde \Pi_{2d}$ is stable under the conjugation action of $\sym_d$. 
By the same reasoning as in the proof of Corollary \ref{cor:dim=num.of.mutlidigraph}, we obtain

\begin{prop} \rm 
The set 
$$\left\{ \sum_{\pi \in E(G)} D_\pi \, \Big | \, G \in \phi_d\left(\widetilde \Pi_{2d}\right)\right\}$$  
forms  a basis for $\Cent_{\Br_d(\delta)}(\cor\sym_d)$, the centralizer of $\cor\sym_d$ in $\Br_d(\delta)$.  In particular, the dimension of $ \Cent_{\Br_d(\delta)}(\cor\sym_d)$ is independent of the choice of $\delta$.
\end{prop}
We will study the set $\phi_d\left(\widetilde \Pi_{2d}\right)$ further in the last subsection.

\vskip1em
\subsection{Orthogonal groups} \label{og}
Let $\cor$ be an infinite field and let
$V$ be a $\cor$-vector space with a fixed basis $\{ v_i \in V \, | \,  i=1, 2, \dots , n \}$. 

Let $(\cdot \hspace{1pt}, \cdot)_q$ and $(\cdot \hspace{1pt}, \cdot)_{q'}$ be symmetric bilinear forms on $V$ given by 
\eqn 
(v_i,v_j)_q =\delta_{i,j} \quad \text{and} \quad (v_i,v_j)_{q'} =\delta_{i,\overline{j}},
\eneqn
respectively, where $\overline{j}=n+1-j$. 
Then the \emph{orthogonal groups} $O(n, q)$  and $O(n, q')$  are given by 
\begin{align*}
\begin{split}
O(n, q): & =\{ f \in GL_n(\cor) \, | \, (fv, fw)_q=(v, w)_q ~ \text{for all} \  v, w \in V \} =\{ f \in GL_n(\cor) \, | \, f^Tf=I_n \} 
\end{split}
\end{align*}
and
\begin{align*}
\begin{split}
O(n, q'): & =\{ g \in GL_n(\cor) \, | \, (gv, gw)_{q'}=(v, w)_{q'} ~ \text{for all} \  v, w \in V \} =\{ g \in GL_n(\cor) \, | \, g^T I'_n g=I'_n \} 
\end{split}
\end{align*}
respectively, where $I_n$ is the identity matrix in $GL_n(\cor)$ and $I'_n$ is the  $n \times n$ permutation matrix 
\begin{displaymath}
I_n'=
\begin{pmatrix}
  &  &  &  & 1 \\
  &  &  & 1 &  \\
  &  & \reflectbox{$\ddots$} &  &  \\
 1 &  &  &  &  
\end{pmatrix}.
\end{displaymath}

\vskip 1em

There is an  algebra homomorphism  $\Psi_{n,d}^{q'} : \Br_d(n)^\opp \to \End(V^{\tens d})$  given  by 
\eqn
&&\Psi_{n, d}^{q'}(s_j)( v_{i_1} \tens \cdots \tens v_{i_d})= v_{i_1} \tens \cdots \tens v_{i_{j-1}} \tens v_{i_{j+1}} \tens v_{i_{j}} \tens v_{i_{j+2}} \tens \cdots \tens v_{i_d}, \ \text{and} \\
&&\Psi_{n, d}^{q'}(e_j) (v_{i_1} \tens \cdots \tens v_{i_d}) =\delta_{i_j, \overline{i}_{j+1}} ~ v_{i_1} \tens \cdots \tens v_{i_{j-1}} \tens \Bigg( \displaystyle \sum_{k=1}^{n} v_{k} \tens v_{\overline{k}} \Bigg) \tens v_{i_{j+2}} \tens \cdots \tens v_{i_d}.
\eneqn

\begin{theorem} \rm (\cite[Theorem 1.2 (b)]{DH}\rm) \label{q'}
Let $\cor$ be an infinite field with $\ch(\cor)\neq 2$. 
Then the image $\Psi_{n, d}^{q'}(\Br_d(n)^\opp)$ is equal to the space $\End_{\cor O(n, q')}(V^{\tens d})$. 
Moreover  if $n \ge d$, then  the homomorphism $\Psi_{n, d}^{q'}$ is injective.
\end{theorem}

From Theorem \ref{thm:dim} and Theorem \ref{q'} we have 
\begin{theorem} \rm \label{dimensionswdo}
Let $\cor$ be an infinite field with $\ch(\cor)\neq 2$. 
Assume that   either $\ch(\cor) =0$ or $\ch(\cor) > d$. 
If  $n \ge d$, then
$$\dim_{\cor}P^d(M_n(\cor))^{O(n,q')}=\dim_{\cor}\Cent_{\Br_d(n)}(\cor\sym_d).
$$
In particular, if $\ch(\cor) =0$, then the sequence $\left( \dim_{\cor}P^d(M_n(\cor))^{O(n,q')} \right)_{n =1}^{\infty}$ is  stable for $n \ge d$.
\end{theorem}

\vskip 1em
Assume that $\ch(\cor) \neq 2$ and 
$\displaystyle \sqrt{-1} \in \cor$. 
Then the map $\varphi_{O}^{ } : V \rightarrow V$ given by
\begin{displaymath}
\varphi_{O}^{ }(v_i)= \left\{ \begin{array}{ll}
v_i +  \sqrt{-1}v_{\overline i} & \text{ if } 1 \le i \le \frac{n}{2} \\
-\dfrac{ \sqrt{-1}}{2} v_i + \dfrac{1}{2} v_{\overline i} & \text{if } ~ i > \frac{n+1}{2} \\
v_i & \text{ if } i=\frac{n+1}{2}
\end{array} \right.
\end{displaymath}
satisfies 
\eq  \label{eq:isometry}
(\varphi_{O}^{ }(v), \varphi_{O}^{ }(w))_{q'}=(v,w)_{q}\qquad \text{for} \ v,w\in V.
\eneq
In other words, $\varphi_{O}^{ }$ is an isometry between the quadratic spaces $(V, (\cdot \hspace{1pt}, \cdot)_{q})$ and $(V, (\cdot \hspace{1pt}, \cdot)_{q'})$.
Actually the assumption $\sqrt{-1} \in \cor$ is necessary if we want to have isometries between $(V, (\cdot \hspace{1pt}, \cdot)_{q})$ and $(V, (\cdot \hspace{1pt}, \cdot)_{q'})$ for sufficiently many different dimensions.
\begin{prop} \rm 
Let $\cor$ be  a field with $\ch(\cor)\neq 2$. If $\dim_\cor V \equiv 2 \ \text{or} \ 3 \pmod 4$, and the quadratic spaces $(V, (\cdot \hspace{1pt}, \cdot)_{q})$ and $(V, (\cdot \hspace{1pt}, \cdot)_{q'})$ are isometric to each other,
then $\sqrt{-1}\in \cor$.
\end{prop}
\begin{proof}
Assume that $n = 4k+2$ (respectively, $4k+3$) for some $k \in \Z_{\ge 0}$.
One can show that matrix $I_n'$ is congruent to a diagonal matrix $D_n$ whose entries are $2k+1$ many $-1$'s  and  $2k+1$ many  (respectively $2k+2$ many) $1$'s, by  a basis change similar to $\varphi_O^{}$.
Hence $I_n$ is congruent to $I'_n$, then $I_n$ is congruent to $D_n$. 
In particular, $1=\det(I_n)$ and $-1=\det(D_n)$ belong to the same square class of $\cor$ and hence 
$\cor$ contains $\sqrt{-1}$, as desired.
\end{proof}

Assume that $\cor$ is an infinite field such that $\sqrt{-1} \in \cor$ and $\ch(\cor)\neq 0$. Set 
\eq \label{eq:PhiOconj}
\Psi_{n, d}^{q}(D_\pi)= (\varphi_{O}^{-1})^{\tens d} \circ \Psi_{n, d}^{q'}(D_\pi) \circ \varphi_{O}^{\tens d}\quad \text{ for } D_\pi \in \Br_d(n)^\opp.
\eneq
Then we have an analogue of Theorem \ref{q'}:  
the image $\Psi_{n,d}^q(\Br_d(n)^\opp)$ of the algebra homomorphism $\Psi_{n,d}^q$ is equal to the space $\End_{\cor O(n, q')}(V^{\tens d})$, and  the homomorphism $\Psi_{n, d}^{q'}$ is injective, provided  $n \ge d$ (Cf. \cite[Proposition 2.8]{GN}).

From Theorem \ref{thm:dim}, we obtain
\begin{theorem} \rm \label{dimensionswdo}
Let $\cor$ be an infinite field with $\ch(\cor)\neq 2$ and $\displaystyle \sqrt{-1}  \in \cor$. 
Assume that   either $\ch(\cor) =0$ or $\ch(\cor) > d$. 
If  $n \ge d$, then
$$\dim_{\cor}P^d(M_n(\cor))^{O(n,q)}=\dim_{\cor}\Cent_{\Br_d(n)}(\cor\sym_d).
$$
In particular, if $\ch(\cor) =0$, then the sequence $\left( \dim_{\cor}P^d(M_n(\cor))^{O(n,q)} \right)_{n =1}^{\infty}$ is  stable for $n \ge d$.
\end{theorem}

\begin{remark}
One can check that 
\eqn
&&\Psi_{n, d}^{q}(s_j) ( v_{i_1} \tens \cdots \tens v_{i_d} )= v_{i_1} \tens \cdots \tens v_{i_{j-1}} \tens v_{i_{j+1}} \tens v_{i_{j}} \tens v_{i_{j+2}} \tens \cdots \tens v_{i_d}, \ \text{and} \\
&&\Psi_{n, d}^{q}(e_j)  (v_{i_1} \tens \cdots \tens v_{i_d})=\delta_{i_j, i_{j+1}} v_{i_1} \tens \cdots \tens v_{i_{j-1}} \tens \Bigg( \displaystyle \sum_{k=1}^{n}  v_k \tens v_{k} \Bigg) \tens v_{i_{j+2}} \tens \cdots \tens v_{i_d}.
\eneqn
\end{remark}

\subsection{Symplectic groups} \label{sg}

Let $\cor$ be an infinite field and let $V$ be a $\cor$-vector space with a fixed basis $\{ v_i \in V \, | \,  i=1, 2, \dots , n \}$. 
Through this subsection, we assume that $n=2m$ for some $m \in \Z_{\ge 1}$.

Let $(\cdot \hspace{1pt}, \cdot)_s$ be nondegenerate skew-symmetric bilinear forms on $V$ given by 
\begin{displaymath}
(v_i,v_j)_s = \left\{ \begin{array}{ll}
1 & \text{ if } j=i+m \text{ and } 1 \le i \le m \\
-1 & \text{ if } j=i-m \text{ and } m+1 \le i \le 2m \\
0 & \text{ otherwise }
\end{array} \right.
\end{displaymath}

The \emph{symplectic group} $Sp_n(\cor)$ is the subgroup of $GL_n(\cor)$ given by
\eqn
Sp_n(\cor):
 =\{ f \in GL_n(\cor) \, | \, f^TJ_nf=J_n \},
\eneqn
where 
\eqn
J_n= 
\left( \begin{array} {c|c} 
			O & I_m \\
			\hline
			-I_m & O
\end{array} \right). 
\eneqn

There is an algebra homomorphisms $\Psi_{n, d}^{s} : \Br_d(-n)^\opp \to \End(V^{\tens d})$ given by 
\eqn
&&\Psi_{n, d}^{s}(s_j)( v_{i_1} \tens \cdots \tens v_{i_d})=  - v_{i_1} \tens \cdots \tens v_{i_{j-1}} \tens v_{i_{j+1}} \tens v_{i_{j}} \tens v_{i_{j+2}} \tens \cdots \tens v_{i_d}, \ \text{and} \\
&&\Psi_{n, d}^{s}(e_j) (v_{i_1} \tens \cdots \tens v_{i_d}) \\
&&\hspace{1cm}= (v_{i_j},v_{i_{j+1}})_s  ~ v_{i_1} \tens \cdots \tens v_{i_{j-1}} \tens \Bigg( \displaystyle \sum_{k=1}^{m} v_{m+k} \tens v_k - v_k \tens v_{m+k} \Bigg) \tens v_{i_{j+2}} \tens \cdots \tens v_{i_d}. 
\eneqn

\vskip1em
\begin{theorem} \rm(\cite[Proposition 1.3, Theorem 1.4]{DDH}\rm) \label{sp}
Let $\cor$ be an  infinite field. 
Then the image
$\Psi_{n, d}^{s}(\Br_d(-n)^\opp)$ is equal to the space $\End_{\cor Sp_n(\cor)}(V^{\tens d})$. 
Moreover,  if $n \ge 2d$, then the homomorphism $\Psi_{n, d}^{s}$ is injective.
\end{theorem}

From Theorem \ref{thm:dim} and Theorem \ref{sp}, we obtain the following theorem.
\begin{theorem} \rm \label{dimensionswdsp}
Let $\cor$ be an infinite field and assume that either $\ch(\cor)=0$ or $\ch(\cor) >d$.
If $n \ge 2d$, then  
$$\dim_{\cor}P^d(M_n(\cor))^{Sp_n(\cor)}=\dim_{\cor}\Cent_{\Br_d(-n)}(\cor\sym_d).
$$
In particular, if $\ch(\cor) =0$, then the sequence $\left( \dim_{\cor}P^d(M_n(\cor))^{Sp_n(\cor)} \right)_{n =1}^{\infty}$ is  stable for $n \ge 2d$.
\end{theorem}

\begin{remark}
Note that $Sp_n(\cor)$ is the group of linear transformations preserving the bilinear form $(\cdot ,\cdot)_s$ on $V$  which is represented by the matrix $J_n$ with respect to the basis we fixed. 
In \cite{DDH}, a different bilinear form $(\cdot ,\cdot)_{s'}$ is considered but it is isometric to $(\cdot ,\cdot)_s$ over an \emph{arbitrary} field. One may take an isometry given by a permutation of the elements in the basis. The algebra homomorphism $\Psi_{n, d}^{s}$ is the  conjugation  of that of \cite{DDH} by the isometry.
\end{remark}

\vskip 1em
\subsection{Centralizer of  $\sym_d$ in $\Br_d(\delta)$ and  disjoint union of directed cycles} \label{directed cycles}

A \emph{directed cycle} is a directed graph whose underlying undirected graph is a cycle. 
Note that  a cycle has at least one edge. 
Let $\LG_{d}^O$ be the subset of $\LG_d$ consisting of directed graphs in which each connected component is a directed cycle. 
 In other words, $\LG_d^O$ is the set of disjoint union of directed cycles whose total number of arrows is $d$.

\begin{lemma} \rm 
The map $\psi_d^{ }$ induces a bijection from $ \widetilde \Pi_{2d}$ to $\LG_{d}^O$.
\end{lemma}
\begin{proof}
Let $\pi \in \widetilde \Pi_{2d}$. 
Fix $a_1 \in [1, 2d]$. 
Let $b_i$ be the other element in the block of $a_i$ and set $a_i:=b_{i-1}'$ unless $b'_{i-1} = a_j$ for some $j <  i$. 
Let $i_0=\text{min}\{ i \ge 2~ | ~ b_{i-1}'=a_j \text{ for some } j<i \}$. 
By the construction, $\{a_k,b_k\}$ is a block of $\pi$ for each $1\le k < i_0$. 
If $b_{i_0-1}'=a_j$ for some $j < i_0$, then $j=1$. 
Indeed, if $j>1$, then $b'_{i_0-1}=a_j=b_{j-1}'$ so that $i_0-1=j-1$, which is a contradiction. 
It follows that the full subgraph of $\widetilde \psi_d(\pi)$ with vertices $\{ a_1, b_1|a_2, b_2 | \cdots | a_{i_0-1}, b_{i_0-1} \}$  forms a directed cycle, because the degree of each vertex of $\widetilde \psi_d(\pi)$ equals two.  
Repeating the procedure, we conclude that $\psi_d(\pi)$ is a multidigraph in which every connected component is a directed cycle. 

It is clear that the restriction of $\psi_d^{-1}$ on $\LG_d^O$ is a map into $\widetilde \Pi_{2d}$.
\end{proof}

Let $\ULG_{d}^O$ be a subset of $\ULG_{d, d}$ consisting of the multidigraphs in which every connected component is a directed cycle. 
Then, since $\sigma$-conjugation on $ \psi_d(\pi)$ is permuting the arrow labels, $\sym_d$-conjugacy classes of $\widetilde \Pi_{2d}$ is in bijection with $\ULG_{d}^O$ under $\phi_d$.

Hence, by the same proof as in Corollary \ref{cor:dim=num.of.mutlidigraph}, we obtain the following theorem. 
\begin{theorem}  \rm 
The dimension of 
the centralizer $\Cent_{\Br_d(\delta)}(\cor\sym_d)$ is equal to $\# \, \ULG_d^O$, the number of  disjoint unions of directed cycles in which the total number of arrows is $d$.
\end{theorem}

\begin{corollary} \label{cor:orthsymp} \rm 
Let $\cor$ be an infinite field with $\ch (\cor) =0$ or $\ch(\cor) > d$. 
\begin{enumerate}
\item
If $\ch \cor \ne 2$ and $n \ge d$, then
$$\dim_\cor P^d(M_n(\cor))^{O(n, q')}= \# \, \ULG_d^O.$$

\item 
If $\ch \cor \ne 2$ and $n \ge d$, and  $\sqrt{-1}\in \cor$, then 
$$\dim_\cor P^d(M_n(\cor))^{O(n, q)}= \# \, \ULG_d^O.$$

\item
If  $n \ge 2d$, then
$$\dim_\cor P^d(M_n(\cor))^{Sp_n(\cor)}= \# \, \ULG_d^O.$$
\end{enumerate}
\end{corollary}

\begin{remark}
If $\cor=\C$, then Corollary \ref{cor:orthsymp} recovers  the \emph{stable behavior of Hilbert series}  by Willenbring in the cases of orthogonal groups $O(n,q)=O_n(\C)$ and symplectic groups $Sp_n(\C)$. More precisely, (2) recovers  \cite[Theorem 4.1]{WillenbringO}, and (2)   together with (3) recover the equality
$$\lim HS(GL_n(\mathbb R)) = \lim HS(GL_m(\mathbb H)) $$
in \cite[Section 1.8]{Willenbring}.
Here the notation $HS(G_0) $ stands for the Hilbert series of the ring $S(\g)^K$, where $(G,K)$ is the \emph{symmetric pair} corresponding to the real form $G_0$, $\g$ is the  Lie algebra of the complex reductive group $G$, and $S(\g)$ denotes the symmetric algebra of $\g$.
When $G_0=GL_n(\mathbb R)$ (respectively, $G_0=GL_m(\mathbb H)$), the corresponding symmetric pair is $(GL_n(\C),O(n,q))$ (respectively, $(GL_{2m}(\C), Sp_{2m}(\C))$) so that  
$\g$ is isomorphic to $M_{n} (\C)$ (respectively, $M_{2m} (\C)$) and the ring $S(\g)^K$ is isomorphic to $P(M_n(\C))^{O(n,q)}$ (respectively, $P(M_{2m}(\C)^{Sp_{2m}(\C)}$).  
The stable limit $\lim HS(G_0)$ is defined as the formal power series whose coefficients are given by the limits of the coefficients of the Hilbert series $HS(G_0)$ as $n \to \infty $.
\end{remark}

\subsection{ Generalized Cycle Types of Brauer diagrams} \label{GCT}
In this subsection, we recall Shalile's description of the $\sigma$-conjugate classes of Brauer diagrams (called the \emph{generalized cycle types} in \cite{Shalile2011,Shalile}) and compare it with the description $\phi_d(\widetilde \Pi_{2d})$ by giving an explicit isomorphism between them.

\vskip 1em
\begin{definition} \rm(\cite[Definition 2.1]{Shalile2011}\rm) \\
\label{def:ULT}
For a diagram $D \in \beta_d$, we get a \emph{string}, sequence of letters $\{U, L, T\}$, as the following process:
\begin{enumerate}[(i)]
\item Starting from a dot in a diagram, move to the other dot which is connected to the original dot by an edge. 
\item If this edge is connected to a dot in the row other than the one in which we started, then mark this edge as ``$T$". \\
		 If the edge is connecting two dots in the same row, then mark it as ``$U$" if it was in the first row and ``$L$" if it was in the second row. 
\item From the dot we arrived by (i), move to the other dot which is in the same column. 
\item Continue (i)$\sim$(iii) until we reach the dot which we started. Here, we get a \emph{string} of letters, composed with $U$, $L$, or $T$. 
\end{enumerate}
From the given process, we get a multiset of strings in $ \{U, L, T\}$. 

From choosing another dot that was not counted in the above process and continuing (i)$\sim$(iv) until there are no remaining dots, one derives a multiset of strings. 

Note that one can have different multiset of strings by starting from a different dot. 
\end{definition}

\vskip 1em 
\begin{remark} \label{rem:GCTs} (Cf. \cite[Remark 2.4]{Shalile2011} )
It is clear that $U$ cannot come right after $U$ in a string of a diagram. 
If $T$ comes after $U$, the edge starts from the bottom to the top row. 
and hence  the next edge starts from the bottom row. 
Therefore, after $U$, another $U$ cannot appear again until $L$ comes after. 
Similarly, after $L$, another $L$ cannot appear again until $U$ comes after. 
\end{remark}

\vskip 1em
\begin{definition} \rm(\cite[Definition 2.5]{Shalile2011}\rm) \\
Let $C=l_1l_2 \cdots {l_{\alpha}}$ be a string in \{U, L, T\} for some $\alpha \in \mathbb{Z}_{\ge 1}$. 
Let the reversing $r$ and shifting $s$ be functions of switching the string of letters $C$ as 
\begin{center}
$r(C)=l_{\alpha}l_{(\alpha -1)} \cdots l_1$ \\
$s(C)=l_2l_3 \cdots l_{\alpha}l_1$
\end{center}
For strings $C_1=l_1 l_2  \cdots l_{\alpha}$ and $C_2=k_1  k_2 \cdots k_{\alpha} $ having the same length, define the relation $\sim$ as
\begin{center}
$C_1 \sim C_2 \quad \text{if and only if} \quad C_2=r^{t_1} s^{t_2}(C_1)$ for some $t_1, t_2 \in \mathbb{Z}_{\geqslant 0}$
\end{center}
Then,  $\sim$ is an equivalence relation because $sr=rs^{\alpha -1}, r^2=$id, and $s^{\alpha}=$id. 

For each $\pi \in \widetilde \Pi_{2d}$, we call the multiset of equivalence classes of strings obtained by the procedure  in Definition \ref{def:ULT}  the  \emph{generalized cycle type (GCT, in short) of $\pi$}. 

\end{definition}


\vskip 1em
\begin{example}
Consider when $d=4$. Let $\delta \in \cor$ and two diagrams $D_1, D_2 \in \Br_d(\delta)$ be 
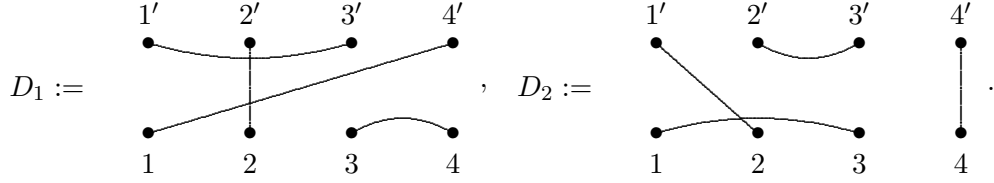
\begin{center}
${\beginpicture 
\setcoordinatesystem units <1.35cm,0.4cm>
\setplotarea x from 0 to 10, y from -1 to 4
\put{$D_1:= $} at 0 1.5
\put{$\bullet$} at  1 0
\put{$\bullet$} at  1 3
\put{$\bullet$} at  2 0
\put{$\bullet$} at  2 3
\put{$\bullet$} at  3 0
\put{$\bullet$} at  3 3
\put{$\bullet$} at  4 0
\put{$\bullet$} at  4 3
\put{,} at 4.3 1.5
\put{$D_2:=$} at 5 1.5
\put{$\bullet$} at  6 0
\put{$\bullet$} at  6 3
\put{$\bullet$} at  7 0
\put{$\bullet$} at  7 3
\put{$\bullet$} at  8 0
\put{$\bullet$} at  8 3
\put{$\bullet$} at  9 0
\put{$\bullet$} at  9 3
\put{.} at 9.3 1.5
\put{1} at 1 -1 
\put{2} at 2 -1
\put{3} at 3 -1
\put{4} at 4 -1
\put{$1'$} at 1 4
\put{$2'$} at 2 4
\put{$3'$} at 3 4
\put{$4'$} at 4 4
\put{1} at 6 -1
\put{2} at 7 -1
\put{3} at 8 -1
\put{4} at 9 -1
\put{$1'$} at 6 4
\put{$2'$} at 7 4
\put{$3'$} at 8 4
\put{$4'$} at 9 4

\plot 1 0 4 3 /
\plot 2 3 2 0 /
\plot 9 3 9 0 /
\plot 6 3 7 0 /
\setquadratic
\plot 1 3 2 2.5 3 3 /
\plot 3 0 3.5 0.5 4 0 /
\plot 7 3 7.5 2.5 8 3 /
\plot 6 0 7 0.5 8 0 /
\endpicture}$
\end{center}
Starting from $1'$ of $D_1$, we proceed as $1' \xrightarrow{ U } 3' \dashrightarrow 3 \xrightarrow{L} 4 \dashrightarrow 4' \xrightarrow{T} 1 \dashrightarrow 1$. 
We did not pass through 2 in the previous process, so starting again from $2'$ to obtain $2' \xrightarrow{T} 2 \dashrightarrow 2'$. 
The resulting generalized cycle type of $D_1$ is $\{ULT, T\}$. 

Similarly, start from $1'$ of $D_2$, we have $1' \xrightarrow{T} 2 \dashrightarrow 2' \xrightarrow{U} 3' \dashrightarrow 3 \xrightarrow{L} 1 \dashrightarrow 1'$,
 and start again from $4$ to obtain  $4 \xrightarrow{T} 4' \dashrightarrow 4$. 
Then the generalized cycle type of $D_2$ is $\{TUL, T\}$. 

Since $ULT=s^{-1}(TUL)$, $D_1$ and $D_2$ have the same generalized cycle type.
\end{example}

\begin{definition} \rm \label{epsilond}
Let $\GCT_d$ be the set of multisets of equivalence classes of strings in $U, L, T$ of length $d$ such that in each string of the multiset
\begin{enumerate}[(i)]
\item no string of the form $UT^iU$ or $LT^iL$  $(i \ge 0)$ appears in any representative, and
\item the number of occurrence of $U$ and that of $L$ are the same.
\end{enumerate}
\end{definition}
By Remark \ref{rem:GCTs}, every generalized cycle type  of an element  $\pi \in \widetilde \Pi_{2d}$ belongs to the set $\GCT_d$.

Define $\rho_d^{ } : \ULG_d^O \longrightarrow \GCT_d^{ }$ as the following:   
\begin{enumerate}[(i)]
\item For a digraph in $\ULG_d^O$, label each vertex as $U$ if $\rightarrow \bullet \leftarrow$, $L$ if $\leftarrow \bullet \rightarrow$, and $T$ if $\rightarrow \bullet \rightarrow$ or $\leftarrow \bullet \leftarrow$.
\item Start from a vertex and make a string of letters of $U$, $L$, $T$'s by following the edges in one orientation. 
\end{enumerate}

Let $f : \LG_{d}^O \to \ULG_d^O$ be the function of  forgetting the labels on arrows and set $$\overline{\rho_d^{ }}:=\rho_d \circ f.$$

\begin{remark} \label{rem:ShalileWillenbring}
One can check easily that the map $\overline{\rho_d^{ }} \circ  \psi_d$ is nothing but the map in Definition \ref{def:ULT} which Shalile defined. 
Also, note that ${\psi_d}|_{\widetilde \Pi_{2d}}$ is the map Willenbring considered in \cite[Section 3]{WillenbringO} under the identification $\widetilde \beta_d$ with the set of fixed point free involutions on $[1,2d]$.
\end{remark}

Hence we obtain a commutative diagram  below: \\
\centerline{
\xymatrix{
\widetilde \Pi_{2d} \ar[r]^{ \psi_d^{ }} \ar[rd]_{\text{Shalile}} & \LG_{d}^O \ar[r]^{f} \ar[d]^{\overline{\rho_d^{ }}} & \ULG_{d}^O \ar[ld]^{\rho_d^{ }} \\
 & \GCT_d^{ } & 
}}

\vskip 1em
We define $\nu_d : \GCT_d \to \ULG_{d,d}^O$ as follows: \\
It is enough to define  $\nu_d(C)$ for an element in $\GCT_d$ of the form $C=l_1l_2 \cdots l_d$. 
\begin{enumerate}
\item If $l_i=T$ for all $1 \le i \le d$, then $\nu_d(C)$ is the directed cycle in Figure \ref{fig:cycleTT}.
\begin{figure}[h]
\begin{tikzpicture}[node distance =1 cm and 1 cm]
\node (a) [point];
\node (b) [below right= of a,yshift=0.5cm, point];
\node (c) [below = of b, point]; 
\node (d) [below = of a,yshift=-1.3cm] {$\cdots$};
\node (e) [left = of c,xshift=-1cm, point];
\node (f) [left = of b,xshift=-1cm, point];

\path (a) edge (b); 
\path (b) edge (c);
\path (c) edge (d); 
\path (d) edge (e); 
\path (e) edge (f);
\path (f) edge (a); 
\end{tikzpicture}
\caption{Directed graph of $\nu_d(TT\cdots T)$}
\label{fig:cycleTT}
\end{figure}
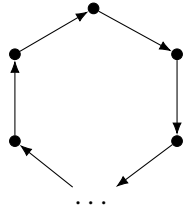
\item We may assume that $l_1=U$ by Definition \ref{epsilond} (ii). 
Construct a planar directed graph $G_i$ inductively as follows:  
\begin{enumerate}
\item Let $G_1:=\bullet^1 \longleftarrow \bullet^2$. 
\item For $2 \le i \le d-1$, 
\begin{enumerate}
\item If $l_i=U$, then $G_i$ is the graph obtained from $G_{i-1}$ by adding one incoming arrow to the rightmost vertex of $G_{i-1}$. 
That is, $G_i=G_{i-1} \longleftarrow \bullet^{i+1}$ where $\bullet^{i}$ is the most recently added vertex in $G_{i-1}$. 

\item If $l_i=L$, then $G_i$ is the graph obtained from $G_{i-1}$ by adding one outgoing arrow to the rightmost vertex of $G_{i-1}$. That is, $G_i=G_{i-1} \longrightarrow \bullet^{i+1}$.

\item If $l_i=T$, then $G_i$ is the graph obtained from $G_{i-1}$ by adding one arrow to the rightmost vertex of $G_{i-1}$ with the same direction as the adjacent one. 
\end{enumerate}

\item For $i=d$, $\nu_d(C)=G_d$ is the graph obtained from $G_{d-1}$ adding an arrow connecting the two extreme vertices from $\bullet^d$ to $\bullet^1$. 

\end{enumerate}
\end{enumerate}

\vskip 1em

\begin{theorem}
The maps $\rho_d$ and $\nu_d$ are inverses to each other. 
\end{theorem}
\begin{proof}
The followings are consequences of the definition of $\nu_d$.
\begin{enumerate}
\item
In the case of (b)-(i), since $l_1=U$, there exists $j < i$ such that $l_j=L$ and $l_{j+1}=l_{j+2}=\cdots =l_{i-1}=T$.
Then, $G_{i-1}=G_{j-1} \longrightarrow \bullet^{j+1} \longrightarrow \bullet^{j+2} \longrightarrow \cdots \longrightarrow \bullet^{i}$. 
So, the local configuration in this step is always $\bullet^{i-1} \longrightarrow \bullet^{i} \longleftarrow \bullet^{i+1}$. 
\item
In the case of (b)-(ii), since $U$ has to appear first before $L$, there exists $j < i$ such that $l_j=U$ and $l_{j+1}=l_{j+2}=\cdots =l_{i-1}=T$. 
Then, $G_{i-1}=G_{j-1} \longleftarrow \bullet^{j+1} \longleftarrow \bullet^{j+2} \longleftarrow \cdots \longleftarrow \bullet^{i}$. 
So, the local configuration in this step is always $\bullet^{i-1} \longleftarrow \bullet^{i} \longrightarrow \bullet^{i+1}$. 
\item
In the case of (c), $l_d$ is either $L$ or $T$ by Definition \ref{epsilond} (i) and the assumption  $l_1=U$.  
\end{enumerate}
By these local characterizations of $\nu_d(C)$,  it is straightforward to see that the compositions $\rho_d$ and $\nu_d$  are the identities. 
\end{proof}

\begin{corollary} \rm
The set of generalized cycle types of elements in  $\widetilde \Pi_{2d}$ is equal to $\GCT_d$. 
\end{corollary}


\providecommand{\bysame}{\leavevmode\hbox to3em{\hrulefill}\thinspace}
\providecommand{\MR}{\relax\ifhmode\unskip\space\fi MR }
\providecommand{\MRhref}[2]{%
  \href{http://www.ams.org/mathscinet-getitem?mr=#1}{#2}
}
\providecommand{\href}[2]{#2}

\end{document}